\documentclass[10pt, english]{article}

\usepackage{latexsym, graphicx}
\usepackage{amsmath}
\usepackage{amsfonts}
\usepackage{amssymb}
\usepackage{hyperref}
\usepackage{babel}
\usepackage[babel=true]{csquotes}
\usepackage{geometry}
\usepackage{color}

\usepackage{amssymb,amsmath,amsfonts,mathrsfs,amsxtra,amsthm}
\usepackage{graphicx,color,subfig}
\usepackage{accents,xspace}
\usepackage{dsfont}
\usepackage{amscd}
\newcommand{\mypapersize}{
\setlength{\textwidth }{16cm}
\setlength{\textheight}{23.5cm}
\setlength{\oddsidemargin}{-0.14cm}
\setlength{\topmargin}{-1.6cm}
}

\mypapersize

\newtheorem{thm}{Theorem}[section]
\newtheorem*{thm*}{Theorem}
\newtheorem{prop}[thm]{Proposition}
\newtheorem{dfn}[thm]{Definition}
\newtheorem{lem}[thm]{Lemma}
\newtheorem{oss}[thm]{Remark}

\newtheorem{cor}[thm]{Corollary}

\usepackage{fancyhdr}
\begin{document}
\title{Part II On strong and non uniform stability of locally damped Timoshenko beam:
Mathematical corrections and gap filling in the proof of Theorem 2.2 in the publication:
\\
\enquote{Uniform stabilization for the Timoshenko beam by a locally distributed
damping} \vspace{0.2cm}
}
\author{Fatiha Alabau-Boussouira\footnote{Laboratoire Jacques-Louis Lions Sorbonne Universit\'{e}, Universit\'{e} de Lorraine} 
}
\bibliographystyle{plain}

\date{July 23, 2023}
\maketitle {}
\tableofcontents
\begin{abstract}
In part I of the rebuttal (see \cite{AB2023I} to the article \cite{SW2003} entitled "Uniform stabilization for the Timoshenko beam by a locally distributed damping" published in 2003, in the journal
{\it Electronic Journal of Differential Equations}, we prove that Lemma 3.6 and Theorem 3.1 are unproved due to major flaws (contradictory assumptions). We also show that Theorem 2.2 and its proofs of strong stability, and non uniform stability in the case of different speeds of propagation, contain several incorrect arguments and several gaps (including missing functional frames). In this part II, we give the precise missing functional frames, fill the gaps and correct several parts contained in the proof of Theorem 2.2 in \cite{SW2003}. We also complete a missing argument (see Remark \ref {ThmA} and Remark \ref{Gaps}) in the proof of Theorem A in \cite{NRL1986} used by \cite{SW2003}. For this we state and prove Proposition \ref{UptoPi} (see also Proposition \ref{UptoPiB} for a general formulation in Banach spaces). We also give the correct formulations, and proofs of strong stability and non uniform stability (in case of different speeds of propagation) for Timoshenko beams.
\end{abstract}
\section{Introduction}
In this part II (see \cite{AB2023I} for part I), we correct the strong stability and non uniform stability results for non equal speeds of propagation presented in the paper ”Uniform stabilization for the Timoshenko beam by a locally distributed damping” published in 2003 in the journal Electronic Journal of Differential Equations written by Abdelaziz Soufyane and Ali Wehbe.

Let us recall that \cite{SW2003} deals with the Timoshenko system introduced by Timoshenko. in the reference [19] of \cite{SW2003} (and as \cite{T1921} in the present paper). This system is denoted by (1.3) in \cite{SW2003}, which is given by:

\enquote{
\begin{equation*}\label{C4.55}
\begin{cases}
\rho u_{tt}=K(u_{x}-v)_{x},  \quad t\geqslant 0, x \in (0,l), \\[1em]
I_{\rho }v_{tt}=EI v_{xx}+K(u_{x}-v) - bv_{t}, \quad t\geqslant 0, x \in (0,l), \hspace{7.5cm}   (1.3)\\
u(t,0)=u(t,l)=0,\quad v(t,0)=v(t,l)=0, \quad t \geqslant 0,\\
u(0,\cdot)=u_0(\cdot),u_t(0,\cdot)=u_1(\cdot), v(0,\cdot)=v_0(\cdot),v_t(0,\cdot)=v_1(\cdot), \quad x \in (0,l).
\end{cases} 
\end{equation*}
}

We recall the setting in which this system is studied in \cite{SW2003}: The coefficients $ \rho, K, I_{\rho}, E, I$ in (1.3) are assumed to be constant and strictly positive, whereas the damping coefficient $b$ in the second wave equation is assumed to be nonnegative and strictly positive  and bounded away from $0$ on a subset $[b_0,b_1]$  of the spatial domain $[0,l]$ in Theorem 2.2 and Theorem 3.1 of \cite{SW2003} (see the assumption (2.4) in \cite{SW2003}).

 In summary,  the Timoshenko system couples two wave equations, through first order and zero order (in space) coupling terms. The second equation in (1.3) is subjected to a linear localized damping term, whereas the first  wave equation is not directly damped. This situation is described as indirect damping of coupled systems.
  
This system is classically reformulated as a first order abstract equation $U'=LU$ associated to the corresponding initial conditions $U(0)=U_0$, where $U=(u,u_2,v,v_2)^T$. The pivot space to be considered is $(L^2(0,l))^4$.
The corresponding unbounded operator $L: D(L) \subset (L^2(0,l))^4 \mapsto (L^2(0,l))^4$ is defined by
\begin{equation}\label{C4.56}
L=\begin{pmatrix}
0 & Id_{H^1_0(0,l)} & 0 & 0 \\
\dfrac{K}{\rho} \partial_{xx} & 0 & - \dfrac{K}{\rho} \partial_{x} & 0 \\
0 & 0 & 0 & Id_{H^1_0(0,l)} \\
\dfrac{K}{I_{\rho} }\partial_{x} & 0 & \dfrac{E\,I}{I_{\rho} }\partial_{xx} - \dfrac{K}{I_{\rho}} & - \dfrac{b}{I_{\rho}}
\end{pmatrix}
\end{equation}
 where the domain of $L$ is defined by:
\begin{equation}\label{C4.57}
D(L)=\Big(\big(H^2(0,l)\cap H^1_0(0,l))\times \big(H^1_0(0,l)\big)\Big)^2.
\end{equation}
The associated energy space is defined as $(H^1_0(0,l)\times L^2(0,l))^2$.

\vskip 2mm

As mentioned in the abstract and in \cite{AB2023I} for part I, \cite{SW2003} contains important flaws and gaps. One theorem, namely the main theorem, cannot be corrected as written, this is Theorem 3.1 and Lemma 3.6 on which Theorem 3.1 is relying.
The purpose of this part II, is to present a correction to the two claims in Theorem 2.2.

We choose to present this corrective part II in a pedagogic way for an unaware public of readers, so that readers can understand on a first step why \cite{SW2003} contains some incorrect mathematical arguments and also mathematical gaps in the proofs of Theorem 2.2 in Section 2, and on a second step how to correct them, and how to fill the gaps in the proofs. An important part is to build the missing functional frames. This has also to be performed it in a precise order, to deal between second order evolution equations and associated first order evolution system through Riemann invariants. This has also to be guaranteed by preserving the necessary stability properties (see for instance Remark \ref{S1C_0notX_0stable} to see that it is not possible for every involved semigroups) and the necessary invariance properties (for the essential spectral radius). These arguments are probably harder to follow for unaware public of readers, even if they are essential for the settings and the proofs.

\begin{oss}\label{Poly}
We deeply correct here the parts of \cite{SW2003} that can be corrected. There exists an alternative, anterior and more general mathematical approach to consider indirect stabilization issues (as well as indirect observability or indirect controllability issues).
We refer the reader to \cite{AB2007} for stronger stability results for indirect damping of Timoshenko systems, namely exponential stability for equal speeds, polynomial stability and strong stability in the case of non equal speed, and optimal general decay rates for arbitrary nonlinear damping in the case of equal speeds, together with some extension in the case of non equal speed. 

These results and the method to prove polynomial stability for smoother initial data in the case of unequal speeds of propagation, is based on the general methodology introduced in 1999-2002 in \cite{A1999, A2000, A2002}. This general methodology is fundamentally relying on the underlying semigroup property and on the proof that for indirectly damped systems or for any system including any hidden indirect damping phenomenon, the energy of the semigroup trajectories satisfies linear non-differential integral inequalities, which involve the degree of smoothness of the initial data, through the use of higher order energies. Timoshenko systems enter in the general frame early developed in \cite{A1999, A2000, A2002}, as shown in \cite{AB2007}. Note that non exponential stability was proved in \cite{ACK2002} for general abstract linearly globally and indirectly damped systems of weakly coupled second order hyperbolic equations, using the invariance of the essential spectral radius by compact operators.

For a general and detailed presentation of the higher order energy and non-differential integral inequality method for general indirect stabilization, indirect controllability, and indirect observability for partial differential equations, we refer the reader to the survey \cite{AB2012}. This allows to understand that a common and coherent mathematical frame can be given for indirect control of coupled systems. 

For lower energy estimates, and some finer properties, we refer the reader to \cite{AB2010$_1$, AB2010$_2$, AB2011}. 
\end{oss}
 
\section{Corrections to flaws and gaps in the proof of Theorem 2.2 in Section 2. of the article \cite{SW2003}}

\subsection{Correction of a first flaw in the proof of the strong stability result in Theorem 2.2  of \cite{SW2003}}
As recalled, system (1.3) can be classically reformulated as a first order system with the underlying semigroup generator $L$ as defined in \eqref{C4.56}-{\eqref{C4.57}. For their proof of the strong stability in Theorem 2.2 at page 4, Soufyane and Wehbe \cite{SW2003} use Benchimol's results (stated in Theorem 2.3 at page 3), showing by contradiction that $L$ has no pure imaginary eigenvalues.

The authors assume that the underlying semigroup has a pure imaginary eigenvalue $i\omega$, where $\omega \in \mathbb{R}^{\ast}$, with a corresponding non vanishing eigenvector $U_1$ and claim that it leads to a contradiction, namely that $U_1 \equiv 0$. The property that $U_1$ is an eigenvector of $L$ associated to the eigenvalue $i\omega$ with $U_1=(u, u_2,v,v_2)^T \in D(L)$ leads to the system

\begin{equation}\label{R16}
\begin{cases}
Ku_{xx} -Kv_x=-\rho\, \omega^2 u, \quad \mbox{in } (0,l),\\
EIv_{xx} +Ku_x -Kv-ib\omega v=-I_{\rho} \omega^2 v, \quad \mbox{in } (0,l),\\
u_2=i\omega u, \quad \mbox{in } (0,l),\\
v_2=i \omega v , \quad \mbox{in } (0,l),\\
u(0)=u(l)=v(0)=v(l)=0.
\end{cases}
\end{equation}
Hence, this implies that $(u,v)$ is a non vanishing solution of the coupled second order system in space denoted by "(2.5)" in \cite{SW2003}:

\enquote{
\begin{equation*}\label{R16b}
\begin{cases}
Ku_{xx} -Kv_x=-\rho\, \omega^2 u, \quad \mbox{in } (0,l),\\
EIv_{xx} +Ku_x -Kv-ib\omega v=-I_{\rho} \omega^2 v, \quad \mbox{in } (0,l), \hspace{7.5cm}   (2.5)\\
u(0)=u(l)=v(0)=v(l)=0.
\end{cases}
\end{equation*}
}

We refer to \cite{AB2023I} for the first flaw in this proof, coming from the fact that the unknown functions $u$, and $v$ are complex-valued and not real-valued as assumed in \cite{SW2003}, so that it is not mathematically correct as in \cite{SW2003}, to multiply the second equation by $v$, perform some integration by parts and take the real and imaginary part of the resulting equation to deduce that the two identities at page 4 of \cite{SW2003} just after "(2.5)" are satisfied.

The proof in \cite{AB2023I} can easily be corrected as follows.  The solutions $(u,v)$ of \eqref{R16b} are in $(H^2(0,l) \cap H^1_0(0,l))^2$. We multiply the first equation of \eqref{R16} (or equivalently of (2.5)) by $\overline{u}$. After elementary integration by parts and the use of the boundary conditions on $u$, this leads to:
\begin{equation}\label{R18}
K\int_0^l |u_x|^2 dx +K \int_0^l \overline{u}v_xdx  -\rho\omega^2 \int_0^l |u|^2dx=0.
\end{equation}
We now multiply the second equation of \eqref{R16} by $\overline{v}$, integrate by part over the interval $[0,l]$ and perform an integration by parts on the first term involving $\int_0^l EI\partial_{xx}v\overline{v}dx=-\int_0^l EI \Big|v_x\Big|^2dx$, due to the homogeneous Dirichlet boundary condition on $v$ at $x=0$ and $x=l$. This leads to the equality

\begin{equation}\label{R17}
EI\int_0^l \Big|\partial_xv\Big|^2 dx + \int_0^l \Big(K-I_{\rho}\omega^2\Big)|v|^2dx +i\omega \int_0^l b|v|^2dx -K \int_0^l \partial_xu \overline{v}dx=0.
\end{equation}

Integrating by parts the last term in \eqref{R17} and using the boundary conditions on $v$, we get:
\begin{equation}\label{R19}
EI\int_0^l |v_x|^2 dx + K \int_0^l \overline{\overline{u}v_x}dx +
\int_0^l \Big(K-I_{\rho}\omega^2\Big)|v|^2dx +i\omega \int_0^l b|v|^2dx =0.
\end{equation}
Adding \eqref{R18} and \eqref{R19}, we obtain
\begin{equation}\label{R20}
K\int_0^l |u_x|^2 dx + EI\int_0^l |v_x|^2 dx + 2K \int_0^l \Re(\overline{u}v_x)dx - \rho\omega^2 \int_0^l |u|^2dx +
\int_0^l \Big(K-I_{\rho}\omega^2\Big)|v|^2dx +i\omega \int_0^l b|v|^2dx =0
\end{equation}
Taking respectively the imaginary part and the real part in \eqref{R20}, we get
\begin{equation}\label{R21}
\begin{cases}
\int_0^l b(x)|v|^2dx =0\,\\
 K\int_0^l |u_x|^2 dx + EI\int_0^l |v_x|^2 dx + 2K \int_0^l \Re(\overline{u}v_x)dx - \rho\omega^2 \int_0^l |u|^2dx +
\int_0^l \Big(K-I_{\rho}\omega^2\Big)|v|^2dx =0.
\end{cases}
\end{equation}

The following assumption in \cite{SW2003}:
\enquote{
\begin{equation}\label{R22}
b \geqslant \overline{b} >0, \quad \mbox{ on } [b_0, b_1] \subset [0, l],
\end{equation}
}

\noindent can now be used in the first equation of \eqref{R21} as in \cite{SW2003} but with the right arguments to prove \eqref{R21}. This allows to rigorously deduce that $v=0$ on $[b_0, b_1]$. Using this last property in the second equation of \eqref{R16}, we deduce that $u_x=0$ on $[b_0, b_1]$. Then, using both $v=u_x=0$ on $[b_0, b_1]$ in the first equation of \eqref{R16}, one rigorously deduce that $u=v=0$ on $[b_0, b_1]$.
Hence, we recover a correct proof to show that $u$ and $v$ satisfies the system of equations (2.6) in \cite{SW2003}, namely that:
\ \\
\enquote{
\begin{equation*}\label{R23}\tag{$2.6$}
\begin{cases}
Ku_{xx} -Kv_x=-\rho\, \omega^2 u, \quad \mbox{in } (0,l),\\
EIv_{xx} +Ku_x -Kv=-I_{\rho} \omega^2 v, \quad \mbox{in } (0,l),\\
v=u=0, \quad \mbox{in } (b_0,b_1),\\
u(0)=u(l)=v(0)=v(l)=0.
\end{cases}
\end{equation*}
}
\subsection{A gap in the proof of strong stability in Theorem 2.2 of \cite{SW2003}}
In \cite{SW2003}, the authors conclude quickly without any proof that $u=v=0$ is the solution of (2.6). Effectively, $u=v=0$ is {\it a} solution of (2.6). The main point is to prove that it is the only solution. Whenever rigorously proved, the property is called a "unique continuation property". 
\subsection{Completion of the proof of strong stability in Theorem 2.2 of \cite{SW2003}}
For pedagogic reasons, we shall detail a way to prove the unique continuation property in this case. For the sake of shortness, let us introduce some notation. We set
\begin{equation}\label{R24}
\alpha^2= \dfrac{\rho\, \omega^2}{K} >0,
\gamma^2= \dfrac{K}{EI} >0,
\beta^2= \dfrac{I_{\rho}\, \omega^2}{EI}>0.
\end{equation}
Then the system of equations (2.6), can be reformulated as
\begin{equation}\label{R25}
\begin{cases}
u_{xx} -v_x=-\alpha^2 u, \quad \mbox{in } (0,l),\\
v_{xx} +\gamma^2u_x=(\gamma^2 - \beta^2) v, \quad \mbox{in } (0,l),\\
v=u=0, \quad \mbox{in } (b_0,b_1),\\
u(0)=u(l)=v(0)=v(l)=0.
\end{cases}
\end{equation}
The smoothness properties of $u$ and $v$ are not given in \cite{SW2003}. Indeed at this stage, one can only assume that $U_1=(u, u_2,v,v_2)^T \in D(L)$. Thus $(u,v) \in (H^2(0,l) \cap H^1_0(0,l))^2$. 
We prove the following proposition.
\begin{prop}\label{Prop1}
The solutions $(u,v)$ of \eqref{R25} are in $(\mathcal{C}^{\infty}([0,l]))^2$. Moreover for any solution $(u,v)$ of \eqref{R25}, $u$ solves the following fourth order differential equation with constant coefficients:
\begin{equation}\label{R26}
u_{xxxx} + (\alpha^2+\beta^2)u_{xx} + \alpha^2(\beta^2 - \gamma^2)u=0\, \quad \mbox{in } (0,l)
\end{equation}

We set 
\begin{equation}\label{R27}
X_-=-\dfrac{(\alpha^2+\beta^2) + \big((\alpha^2-\beta^2)^2 +4 \alpha^2 \gamma^2\big)^{1/2}}{2} <0,
X_+=\dfrac{2\alpha^2(\gamma^2-\beta^2)}{(\alpha^2+\beta^2) + \big((\alpha^2-\beta^2)^2 +4 \alpha^2 \gamma^2\big)^{1/2}}
\end{equation}
Then the general solution of \eqref{R26} is given by
\begin{equation}\label{R28}
u(x)=A_1\cos\big(\big|X_-\big|^{1/2}x\big)+A_2\sin\big(\big|X_-\big|^{1/2}x\big) + \begin{cases}
A_3e^{x\sqrt{X_+}}+A_4 e^{-x\sqrt{X_+}}\mbox{ if } \gamma^2 > \beta^2,\\
A_3\cos\big(|X_+\big|^{1/2}x\big) + A_4\sin\big(|X_+\big|^{1/2}x\big) \mbox{ if } \gamma^2 < \beta^2,\\
A_3x+A_4 \mbox{ if } \gamma^2 = \beta^2,
\end{cases}
\end{equation}
where $A_1, A_2, A_3, A_4$ are arbitrary real constants. 
\end{prop}
\begin{proof}
Using \eqref{R25}, we can easily deduce that $(u,v) \in (H^3(0,l) \cap H^1_0(0,l))^2$, and by induction that $(u,v) \in (H^p(0,l) \cap H^1_0(0,l))^2$ for any integer $p \in \mathbb{N}$. Using classical Sobolev embeddings, this implies that $(u,v) \in (\mathcal{C}^{\infty}([0,l]))^2$.
We differentiate twice the second equation of \eqref{R25} and use the first equation of \eqref{R25} in the resulting equation. This gives
\begin{equation}\label{R29}
v_{xxx}=-\beta^2u_{xx} +\alpha^2(\gamma^2-\beta^2)u.
\end{equation}
On the other hand, differentiating twice the first equation of \eqref{R25}, we get 
\begin{equation}\label{R30}
v_{xxx}=u_{xxxx} +\alpha^2u_{xx}.
\end{equation}
Using \eqref{R30} in \eqref{R29}, we deduce that $u$ solves \eqref{R26}.
The characteristic equation:
\begin{equation}\label{R31}
r^4 + (\alpha^2+\beta^2)r^2 + \alpha^2(\beta^2 - \gamma^2)=0
\end{equation}
associated to \eqref{R26} has the solutions $(r_1,r_2,r_3,r_4)$ given respectively by
\begin{equation}\label{R32}
r_1=-i\sqrt{\dfrac{(\alpha^2+\beta^2)+ \big(\alpha^2-\beta^2)^2 +4 \alpha^2 \gamma^2\big)^{1/2}}{2}}, r_2=-r_1.
\end{equation}
\begin{equation}\label{R33}
\begin{cases}
r_3=-i\sqrt{\dfrac{(\alpha^2+\beta^2)- \big(\alpha^2-\beta^2)^2 +4 \alpha^2 \gamma^2\big)^{1/2}}{2}}, r_4=-r_3 \mbox{ if }
\gamma^2 < \beta^2, \\
r_3=-\sqrt{\dfrac{-(\alpha^2+\beta^2)+ \big(\alpha^2-\beta^2)^2 +4 \alpha^2 \gamma^2\big)^{1/2}}{2}}, r_4=-r_3 \mbox{ if }
\gamma^2 > \beta^2, \\
r_3=r_4=0 \mbox{ if } \gamma^2 = \beta^2.
\end{cases}
\end{equation}
We deduce that the general form of the solutions of \eqref{R26} is given by \eqref{R28}.
\end{proof}
Thanks to Proposition  \ref{Prop1}, we have
\begin{cor}\label{Cor2}
The only solution $(u,v)$ of (2.6) is $(u,v)=(0,0)$.
\end{cor}
\begin{proof}
Let $(u,v)$ be any solution of (2.6). Then $(u,v)$ solves \eqref{R25}.
Using Proposition \ref{Prop1} we know that $u$ has the general form \eqref{R28}. Using $u\equiv 0$ on $(b_0,b_1)$, we get that $A_1=A_2=A_3=A_4=0$, so that $u \equiv 0$ on $[0,l]$. Using then the first equation of \eqref{R25}, we deduce that $v_x \equiv 0$ on $[0,l]$. Since $v(0)=v(l)=0$, we deduce that $v \equiv 0$ on $[0,l]$.
\end{proof}
\begin{oss}
One can give a shorter proof based on the fact that any solution of \eqref{R26} is real analytic, since the coefficients are constant. On the other hand, since $u \equiv 0$ on $[b_0,b_1]$, we deduce that $u^{(n)}(b_0)=0$ for any integer $n \in \mathbb{N}$. This combined with the real analyticity of $u$ imply that $u \equiv 0$ on $[0,l]$, and thus $v \equiv 0$ on on $[0,l]$.
\end{oss}

\section{On some gaps in the proof of non-uniform stability in Theorem 2.2 of \cite{SW2003} in case of distinct speeds of propagation}

The second statement in Theorem 2.2 concerns the property of non-uniform stability when 
\begin{equation}\label{DS}
\dfrac{K}{\rho} \neq \dfrac{EI}{I_{\rho}}.
\end{equation}
Under this hypothesis, Soufyane and Wehbe in \cite{SW2003},  make use of an approach presented in [6]. 

Let us describe shortly the strategy of the proof in \cite{SW2003}.

The authors first define a new operator $L_1=L-D$, where  $D$ is a suitable compact operator in the energy space. Since $D$ is compact in the energy space, the authors conclude that $r_e(e^{tL})=r_e(e^{tL_1})$, where $r_e(B)$ stand for the essential spectral radius of any bounded operator in a Banach space $G$. This allows them to consider the coupled system corresponding to the new first order system $\widehat{Y}'(t)=L_1\widehat{Y}(t)$ (instead of the original one)
where $\widehat{Y}(t)=(u,u_t,v,v_t)^T$. This new system is given by:
\begin{equation}\label{C3.1}
\begin{cases}
u_{tt}=\dfrac{K}{\rho}(u_{x}-v)_{x},  \quad t\geqslant 0, x \in (0,l), \\[1em]
v_{tt}=\dfrac{EI}{I_{\rho }}v_{xx}+\dfrac{K}{I_{\rho }}u_{x}-\dfrac{b}{I_{\rho }}v_{t}, \quad t\geqslant 0, x \in (0,l), \\
u(t,0)=u(t,l)=0,\quad v(t,0)=v(t,l)=0, \quad t \geqslant 0,\\
u(0,\cdot)=u_0(\cdot),u_t(0,\cdot)=u_1(\cdot), v(0,\cdot)=v_0(\cdot),v_t(0,\cdot)=v_1(\cdot), \quad x \in (0,l).
\end{cases} 
\end{equation}
The authors then use the Riemann invariants $(p,\varphi, q, \psi)^T$ associated to the first order system $Y'(t)=L_1Y(t)$ with initial condition $Y(0)=Y_0=:(u_0,u_1,v_0,v_1)^T$ to reformulate the above second order system as a first order hyperbolic system in the new variables.
The Riemann invariants are given by
\begin{equation}\label{C3.2}
\begin{cases}
p=-\sqrt{\dfrac{K}{\rho}}u_x + u_t, \quad q= \sqrt{\dfrac{K}{\rho}}u_x + u_t,\\
\hat{\varphi}=-\sqrt{\dfrac{EI}{I_{\rho}}}v_x + v_t, \quad \psi= \sqrt{\dfrac{EI}{I_{\rho}}}v_x + v_t.
\end{cases} 
\end{equation}
This leads to the first order system:
\begin{equation}
\partial_tZ + \hat{K}\partial_xZ +C Z=0
\end{equation}
where $Z=(p, \hat{\varphi}, q, \psi)^T$ with the homogeneous Dirichlet boundary conditions for $(p+q)$ and $\hat{\varphi} + \psi$ given at the top of page 5, and $\hat{K}$ is a diagonal square matrix of order $4$ depending on the physical constants involved in (1.3) of \cite{SW2003}, and $C$ is a square matrix of order $4$ depending on the damping coefficient $b$ and the physical constants involved in (1.3). 

\vskip 2mm

Let us define the operator:
\begin{equation}\label{C3.3}
S_1:=-(\hat{K}\partial_x +C)
\end{equation}
associated to homogeneous Dirichlet boundary conditions for $(p+q)$ and $\hat{\varphi} + \psi$ at both ends $x=0$ and $x=l$:
\begin{equation}\label{C3.4}
(p+q)(0)=(p+q)(l)=(\hat{\varphi} + \psi)(0)=(\hat{\varphi} + \psi)(l)=0.
\end{equation}
Since the two speeds of propagation are assumed to be distinct in the second statement of Theorem 2.2 in \cite{SW2003}, the authors claim that the essential spectral radius of of $L_1$ is equal to the essential spectral radius of the operator
 $L_2=-(\hat{K}\partial_x + C_0)$ associated to the boundary conditions \eqref{C3.4}, where $C_0$ is the diagonal matrix extracted from $C$ (setting all extra diagonal coefficients of $C$ to $0$). It is easy to check that the diagonal matrix $\hat{K}$ satisfies the assumptions in [6], so one can apply [6]. However the application of the appropriate theorem of [6], does not give right away the claimed result (see the second statement of Theorem 2.2 and its proof in \cite{SW2003}). Two gaps are present in the proof of this claim as explained in \cite{AB2023I}. For understanding where these gaps are appearing, we give some details about the property which is proved in [6], and the incorrect application of [6] in \cite{SW2003}.
 \begin{oss}\label{Rk4}
 To describe where the two gaps stand, we briefly recall here (a simplified) version of the framework introduced by Neves-Ribeiro-Lopes in [6], with some changes of notation to avoid confusion. Let $M$ be a given diagonal matrix with eigenvalues ordered as strictly $N$ positive ones and $n-N$ strictly negative ones, where the matrices $D, F, G, \widehat{E}$ have the appropriate lines and columns and $n\,, N \in \mathbb{N}^{\ast}$, with $n >N$.
Neves-Ribeiro-Lopes's article [6] compares the essential spectral radius of a first hyperbolic system written in a block form as 
 \begin{equation}\label{C3.5}
 \partial_t(u,v)^T + M\partial_x(u,v)^T + C(u,v)^T=0, \mbox{ on } [0,l],
 \end{equation}
 subjected to the dynamic boundary condition
 \begin{equation}\label{C3.6}
 \dfrac{d}{dt}\Big(v(t,l)-Du(t,l)
 \Big)=Fu(t,l) + Gv(t,l), 
 \end{equation}
and a boundary condition of the form
\begin{equation}\label{C3.7}
 u(t,0)=\widehat{E}v(t,0)\
 \end{equation}
\noindent to the second first hyperbolic system (also written in a block form):
\begin{equation}\label{C3.8}
 \partial_t(u,v)^T + M\partial_x(u,v)^T + C_0(u,v)^T=0, \mbox{ on } [0,l],
 \end{equation}
\begin{equation}\label{C3.9}
 v(t,l)=Du(t,l), 
 \end{equation}
and \eqref{C3.7}. The column vector functions $u$ and $v$ are respectively of size $N$ and $n-N$.

The authors of [6] denote by $C_{0}$ is the diagonal matrix extracted from $C$ (replacing the extra diagonal coefficients by $0$) and assume that the positive as well as the negative eigenvalues of the diagonal matrix $M$ are distinct. Under this assumption,
they prove that the essential spectral radius $r_e^C$ of the semigroup generated by the operator $-(M\partial_x +C)$ subjected \eqref{C3.6}-\eqref{C3.7} is equal to the essential spectral radius $r_e^{C_{0}}$ of the semigroup generated by $-(M\partial_x +C_0)$ subjected to \eqref{C3.7} and  \eqref{C3.9} (see problem (II.4)) in [6] as briefly mentioned in \cite{SW2003} at page 5).
\end{oss}

\begin{oss}\label{Gaps}
Let us now go back to the proof of the second statement in Theorem 2.2 of \cite{SW2003} in the frame of [6].
Let the matrices $D$ and $\widehat{E}$ be defined below in \eqref{C4.3}, the matrices $F$ and $G$ be defined in \eqref{C4.4}, whereas the matrices $\widehat{K}$ and $C$ are respectively defined in  \eqref{C4.1} and  \eqref{C4.2}. 
 Let us denote by $S_2$ the operator $-(\hat{K}\partial_x +C)$ subjected to the boundary conditions \eqref{C3.6}-\eqref{C3.7}. Recall that  $L_2$ is defined by $L_2=-(\hat{K}\partial_x + C_0)$ and the associated boundary conditions \eqref{C3.7} and \eqref{C3.9}. Note now, that $S_1$ is the operator $-(\hat{K}\partial_x +C)$ subjected to the boundary conditions \eqref{C3.7} and \eqref{C3.9}. Thanks to Neves-Ribeiro-Lopes [6], $r_e(e^{tS_2})=r_e(e^{tL_2})$. However, the authors of [1] assert at page 5 just after the definition of the matrix $C$, that $r_e(e^{tL_1})=r_e(e^{tL_2})$.
 Soufyane and Wehbe do not prove that $r_e(e^{tS_1})=r_e(e^{tS_2})$. This is a first gap in their proof. The second gap is that the authors do not prove that $r_e(e^{tL_1})=r_e(e^{tS_1})$. We complete these two gaps in the next section. This also requires to precise the functional setting in which these different unbounded operators should be considered, and which operators are really involved for the filling of these two gaps. This also requires to prove the invariance of the essential spectral radius under the action of certain mutiplicative transformation as will be seen in the sequel. Indeed we prove a more general result stated in Proposition \ref{UptoPi} (see also Proposition \ref{UptoPiB}). This more general result is applied twice in the present paper, a first time to complete a missing argument in the proof of Theorem A in \cite{NRL1986}, and a second time to fill a third gap in \cite{SW2003}.
 \end{oss}

\section{Complete proof of non-exponential stability of Timoshenko system for different speeds of propagation}
In this Section, we explain how to fulfill the two gaps as explained in Remark \ref{Gaps} and  some other lacks and mathematical incorrections  in the proof of the second statement in Theorem 2.2 of \cite{SW2003}. Thus we provide a complete proof for the non exponential stability of (1.3) in the case of different speeds of propagation.

For this, we recall below some classical definitions and results of spectral theory and strongly continuous semigroup properties.

\begin{oss}\label{StandardDef}
We recall that from standard definitions, the spectrum of a given operator $S$ acting on a Banach space $G$ is denoted by $\sigma(S)$ and is defined by $\sigma(S)=:\mathbb{C} \backslash \rho(S)$, where the resolvent set $\rho(S)$ is defined as
$\rho(S)=\{ \lambda \in \mathbb{C}, (\lambda Id_G - S) \mbox{ is invertible}\}$. Hence $\sigma(S)=\{ \lambda \in \mathbb{C}, (\lambda Id_G - S) \mbox{is not invertible}\}$. 
In particular the set of eigenvalues of $S$, defined by $\sigma_p(S)=:\{\lambda \in \mathbb{C},(\lambda Id_G -S) \mbox{ is not injective} \}$ is included in $\sigma(S)$. In the general case, one has $\sigma(S)=\sigma_p(S) \cup \sigma_c(S) \cup \sigma_r(S)$ where $\sigma_c(S)$ stands for the continuous spectrum, and $\sigma_r(S)$ for the residual spectrum of $S$ (see \cite{Pazy}).

The growth bound $\omega$ of a strongly continuous semigroup $(e^{tS})_{t \geqslant 0}$ generated by an infinitesimal unbounded operator $S$ acting on the Banach space $G$ is defined as
\begin{equation}\label{OMG}
\omega=\inf_{t>0}\Big(\dfrac{\log{\big(||e^{tS}||_{L(G)}\big)}}{t}\Big).
\end{equation}
The spectral radius of the bounded operator $e^{tS}$ is defined as 
\begin{equation}\label{SPR}
r(e^{tS})=\sup\{|\mu|, \mu \in \sigma(e^{tS})\}.
\end{equation}
We recall that:
 \begin{equation}\label{SRW}
r(e^{tS})=e^{\omega t} \quad \forall \ t \geqslant 0,
\end{equation}
and that the semigroup $(e^{tS})_{t \geqslant 0}$ is exponentially stable in $G$ or equivalently uniformly stable in $G$, if and only if $\omega<0$.
\end{oss}

We also need to recall the definition of the essential spectral radius and the Hadamard formula for a characterization of the essential spectral radius, that will be useful in the sequel.
\begin{oss}\label{G}
Let $G$ be any Banach space. In the sequel, we will need to deal with the norm on the quotient space of the set of all bounded linear operators on $G$ by the set of all compact operators on $G$. Let us recall the definition of this norm and a property of the essential radius of any bounded operator on $G$.
We recall that if we denote by $\mathcal{K}(G)$ the space of compact operators on $G$. The norm  defined on the quotient space $\mathcal{L}(G)/ \mathcal{K}(G)$  and denoted by $|| \cdot||_{\mathcal{K}(G)}$ 
is given by:
\begin{equation}\label{C4.23}
|| T ||_{\mathcal{K}(G)}=\inf\{ ||S+K||_{\mathcal{L}(G)}, K \in \mathcal{K}(G)\},
\end{equation}
for any $T$ in $\mathcal{L}(G)$ (see \cite{Y1965}, \cite{N1970}, \cite{K1980}),
and the following Hadamard formula holds:
\begin{equation}\label{C4.24}
r_e(T)=\lim_{n \rightarrow \infty}||T^n||_{\mathcal{K}(G)}^{1/n}, \quad \forall \ T \in \mathcal{L}(G).
\end{equation}
We also recall the classical inequality (which can be easily deduced from the Hadamard formula for $r(T)$):
\begin{equation}\label{relr}
r_e(T) \leqslant r(T) \quad \forall \ T \in \mathcal{L}(G).
\end{equation}
\end{oss}

\subsection{On some needed results on the essential spectral radius}
In \cite{AB2023I}, we pointed and shortly described  two gaps concerning the part of Theorem 2.2 of \cite{SW2003} devoted to the non exponential stability.
For filling these two gaps, we will also need to recall some properties of the essential spectral radius. For the sake of completeness and for pedagogic reasons, we shall also prove further properties of this radius, that we need as for the next Proposition.
\begin{prop}\label{Inv}
Let $(G, ||\cdot||_G)$ and $(F, ||\cdot||_F)$ be two isomorphic Banach spaces. We denote by $J \in \mathcal{L}(F,G)$ the isomorphism between $F$ and $G$. Let $R$ be any bounded operator on $G$. Then the bounded operator on $F$ defined by $S=J^{-1}RJ$ satisfies
\begin{equation}\label{Iso}
r_e(S)=r_e(R)
\end{equation}
\end{prop}
\begin{proof}
\begin{equation}\label{G4.25}
r_e(S)=\lim_{n \rightarrow \infty}||S^n||_{\mathcal{K}(F)}^{1/n}= \lim_{n \rightarrow \infty}||J^{-1}R^n J||_{\mathcal{K}(F)}^{1/n}.
\end{equation}
On the other hand, noticing that for any $K \in \mathcal{K}(G)$, $J^{-1}K J \in \mathcal{K}(F)$, we have for any $n \in \mathbb{N}^{\ast}$
\begin{equation}\label{G4.26}
||J^{-1}R^n J||_{\mathcal{K}(F)} \leqslant || J^{-1}R^n J + J^{-1}K J||_{\mathcal{L}(F)} \leqslant ||J^{-1}||\; ||R^n +K ||_{\mathcal{L}(G)}
\;||J|| , \quad \forall \ K \in \mathcal{K}(G).
\end{equation}
Hence, taking the infimum over all $K \in \mathcal{K}(G)$, we obtain
\begin{equation}\label{G4.27}
||J^{-1}R^n J||_{\mathcal{K}(F)} \leqslant  ||J^{-1}|| \;||R^n||_{\mathcal{K}(G)}\;||J||,
\end{equation}
so that
$$
||J^{-1}R^n J||_{\mathcal{K}(F)}^{1/n} \leqslant  (||J^{-1}||\; ||J|| )^{1/n}\; ||R^n||_{\mathcal{K}(G)}^{1/n},
$$
Taking the limit as $n$ goes to $\infty$, we get
\begin{equation}\label{G4.28}
r_e(S) \leqslant r_e(R).
\end{equation}
In a similar way and since for any $K \in \mathcal{K}(F)$, one has $JKJ^{-1} \in \mathcal{K}(G)$, we get
\begin{equation}\label{G4.29}
||R^n||_{\mathcal{K}(G)} \leqslant || J(J^{-1}R^n J)J^{-1} + JKJ^{-1}||_{\mathcal{L}(G)} \leqslant ||J||\; || J^{-1}R^nJ +K ||_{\mathcal{L}(F)}\; ||J^{-1}||, \quad \forall \ K \in \mathcal{K}(F).
\end{equation}
Hence, taking the infimum over all $K \in \mathcal{K}(F)$, we obtain
\begin{equation}\label{G4.30}
||R^n||_{\mathcal{K}(G)} \leqslant ||J||\; ||J^{-1}R^nJ ||_{\mathcal{K}(F)}\; ||J^{-1}||,
\end{equation}
so that
\begin{equation}\label{Gb4.30}
||R^n||_{\mathcal{K}(G)}^{1/n} \leqslant (||J^{-1}||\; ||J|| )^{1/n}\;  ||J^{-1}R^nJ ||_{\mathcal{K}(F)}^{1/n}, \quad \forall n \in \mathbb{N}^{\ast}.
\end{equation}
Taking the limit as $n$ goes to infinity we deduce that
\begin{equation}\label{G4.31}
r_e(R) \leqslant r_e(S).
\end{equation}
We conclude using \eqref{G4.28} and \eqref{G4.31}.
\end{proof}

\begin{prop}\label{UptoPi}
Let $(G, \langle\cdot,\cdot\rangle_G, \langle \cdot,\cdot \rangle_G, ||\cdot||_G)$  be a Hilbert space of infinite dimension, and $G_0 \subset G$ a closed vectorial subspace of $G$. We denote by $P$ the orthogonal projection of $G$ on $G_0$.  We recall that $P \in \mathcal{L}(G)$. Then we have the following property
\begin{equation}\label{PIradius}
r_e(R\,P)=r_e(R), \quad \forall \ R \in \mathcal{L}(G_0).
\end{equation}
\end{prop}
\begin{proof}

We denote by $\mathcal{K}(G)$ and respectively $\mathcal{K}(G_0)$ the set of all linear compact operators on $G$, respectively on $G_0$. Let $R \in \mathcal{L}(G_0)$ be given arbitrarily.
Let us first prove by induction that for any $n \in \mathbb{N}^{\ast}$, the following property holds:
\begin{equation}\label{Rn}
(R\,P)^n=R^n\, P. 
\end{equation}
The property \eqref{Rn} holds for $n=1$. Let us assume that it holds for $n$. We have, thanks to our induction hypothesis
$$
(R\,P)^{n+1}= (R\,P)^n (R\,P)= R^n\, P (R\,P)
$$
Since $(R\,P)G \subset G_0$ and $P_{|G_0}=Id_{G_0}$, we deduce that $PR\,P=R\,P$. Hence, we have
$$
(R\,P)^{n+1}= R^n\, R\,P= (R)^{n+1}\,P,
$$
so that \eqref{Rn} holds for any positive integer $n$. 
Let us now first prove that $r_e(R) \leqslant r_e(R\,P)$. For this we will use the characterization \eqref{C4.24} of the essential spectral radius. Hence, we have
$$
r_e(R)=\lim_{n \rightarrow \infty}||R^n||_{\mathcal{K}(G_0)}^{1/n}
$$
Let $K \in \mathcal{K}(G)$ be given arbitrarily. We set $\widehat{K_0}=P\, K_{|G_0}$. Then it is easy to check that $\widehat{K_0} \in \mathcal{K}(G_0)$. Thus we have for all $n \in \mathbb{N}^{\ast}$:
\begin{equation}\label{RnKO}
||R^n||_{\mathcal{K}(G_0)} \leqslant ||R^n P_{|G_0} + \widehat{K_0}||_{\mathcal{L}(G_0)} =\sup_{U \in G_0, U \neq 0} \Big(\dfrac{||R^nPU + P K U||_{G_0}}
{||U||_{G_0}}\Big).
\end{equation}
Since for all $U \in G$ and all $n \in \mathbb{N}^{\ast}$, $R^nPU \in G_0$, we have
$$
R^nPU + P K U=PR^nPU + P K U \in G_0 \quad \forall \ U \in G
$$
Using this identity in \eqref{RnKO}, we obtain
\begin{multline*}
||R^n||_{\mathcal{K}(G_0)} \leqslant \sup_{U \in G_0, U \neq 0} \Big(\dfrac{||PR^nPU + P K U||_{G}}
{||U||_{G}}\Big) \leqslant 
\sup_{U \in G, U \neq 0} \Big(\dfrac{||P\big(R^nPU + K U\big)||_{G}}
{||U||_{G}}\Big) \leqslant \\
||P||_{\mathcal{L}(G)} \sup_{U \in G, U \neq 0} \Big(\dfrac{||(R^nP+K)U||_{G}}
{||U||_{G}}\Big)= ||P||_{\mathcal{L}(G)} ||R^nP + K ||_{\mathcal{L}(G)}, \quad \forall \ K \in \mathcal{K}(G), \forall \  n \in \mathbb{N}^{\ast}.
\end{multline*}
Thus, taking the infimum over all $K \in \mathcal{K}(G)$, on both sides of the above inequality and using \eqref{Rn}, we deduce that
$$
||R^n||_{\mathcal{K}(G_0)} \leqslant ||P||_{\mathcal{L}(G)} ||R^nP||_{\mathcal{K}(G)}= ||P||_{\mathcal{L}(G)} ||(RP)^n||_{\mathcal{K}(G)} \quad \forall \ n \in \mathbb{N}^{\ast}, 
$$
so that
$$
||R^n||^{1/n}_{\mathcal{K}(G_0)} \leqslant ||P||^{1/n}_{\mathcal{L}(G)} ||(RP)^n||^{1/n}_{\mathcal{K}(G)} \quad \forall \ n \in \mathbb{N}^{\ast}
$$
Taking the limit as $n$ goes to infinity on both sides of the above inequality, we get
\begin{equation}\label{reinf}
r_e(R) \leqslant r_e(R\,P)
\end{equation}
Let us now  prove the reverse inequality. Let $K_0 \in \mathcal{K}(G_0)$ be given arbitrarily. We set $K=K_0P$. Then it is easy to check that $K \in \mathcal{K}(G)$. Thus we have
\begin{multline*}\label{RnK0}
||(RP)^n||_{\mathcal{K}(G)} = ||R^nP||_{\mathcal{K}(G)}\leqslant ||R^n P+ K||_{\mathcal{L}(G)} = \\
||(R^n + K_0)P||_{\mathcal{L}(G)} \leqslant 
||P||_{\mathcal{L}(G)} ||R^n + K_0||_{\mathcal{L}(G_0)} \quad \forall \ K_0 \in \mathcal{K}(G_0), \forall \ n  \in \mathbb{N}^{\ast}.
\end{multline*}
Taking the infimum over all $K_0 \in \mathcal{K}(G_0)$, we get
\begin{equation*}\label{RnK0}
||(RP)^n||_{\mathcal{K}(G)} \leqslant 
||P||_{\mathcal{L}(G)} ||R^n||_{\mathcal{K}(G_0)}, \forall \  n \in \mathbb{N}^{\ast},
\end{equation*}
so that 
\begin{equation*}
||(RP)^n||^{1/n}_{\mathcal{K}(G)} \leqslant 
||P||^{1/n}_{\mathcal{L}(G)} ||R^n||^{1/n}_{\mathcal{K}(G_0)},\forall \  n \in \mathbb{N}^{\ast}.
\end{equation*}
Taking the limit as $n$ goes to infinity on both sides of the above inequality, we get
\begin{equation}\label{resup}
r_e(R\,P) \leqslant r_e(R)
\end{equation}
We conclude the proof thanks to \eqref{reinf} and \eqref{resup}.
\end{proof}
\begin{oss}\label{Utopi}
We observe that we used only some properties among the assumptions of Proposition \ref{UptoPi}, so that we can formulate in the next Proposition our result in a more general frame.
\end{oss}
\begin{prop}\label{UptoPiB}
Let $(G, ||\cdot||_G)$  be a Banach space of infinite dimension, and $G_0 \subset G$ a closed vectorial subspace of $G$. Let $P \in \mathcal{L}(G)$ be such that the following properties holds:
\begin{equation}\label{PBanach}
\begin{cases}
P_{|G_0}=Id_{|G_0},\\
PG \subset G_0,
\end{cases}
\end{equation}
Then \eqref{PIradius} holds.
\end{prop}

\subsection{Fulfilling of the first gap in \cite{SW2003} in the proof of non-exponential stability in the case of non equal speed of propagation}

This subsection is devoted to a complete rigorous proof that $r_e(e^{tS_1})=r_e(e^{tS_2})$.

We set $N=2$, $n=4$. In the sequel the column vectorial functions  $u$ and $v$ take values in $\mathbb{R}^2$, whereas $z$ will be a column vector in $\mathbb{R}^2$.
In all the sequel we identify the line vectors, with their corresponding transpose column vectors.

\begin{dfn}
We set
\begin{equation}\label{C4.1}
\widehat{K}=%
\begin{pmatrix}
\sqrt{K/\rho} & 0  & 0 & 0 \\ 
0 & \sqrt{E I / I_{\rho}} & 0  & 0 \\
0  & 0 & -\sqrt{K/\rho} & 0 \\
0  & 0 & 0 & -\sqrt{E I / I_{\rho}} 
\end{pmatrix}%
\end{equation}
  \begin{equation}\label{C4.2}
C=%
\begin{pmatrix}
0 & -\dfrac{K}{2\rho} \, \sqrt{\dfrac{I_{\rho}}{EI} }& 0  & \dfrac{K}{2\rho} \, \sqrt{\dfrac{I_{\rho}}{EI} }\\ 
\dfrac{\sqrt{\rho K}}{2I_{\rho}} & \dfrac{b}{2 I_{\rho}} & - \dfrac{\sqrt{\rho K}}{2I_{\rho}} & \dfrac{b}{2 I_{\rho}} \\
0 &  -\dfrac{K}{2\rho} \, \sqrt{\dfrac{I_{\rho}}{EI} } & 0 & \dfrac{K}{2\rho} \, \sqrt{\dfrac{I_{\rho}}{EI} } \\
 \dfrac{\sqrt{\rho K}}{2I_{\rho}} &  \dfrac{b}{2 I_{\rho}} & - \dfrac{\sqrt{\rho K}}{2I_{\rho}} & \dfrac{b}{2 I_{\rho}}
\end{pmatrix}%
\end{equation}
  \begin{equation}\label{C0}
C_0=%
\begin{pmatrix}
0 & 0 & 0  & 0\\ 
0& \dfrac{b}{2 I_{\rho}} & 0 & 0\\
0 &  0 & 0 & 0 \\
0 &  0 & 0 & \dfrac{b}{2 I_{\rho}}
\end{pmatrix}%
\end{equation}
 \begin{equation}\label{C4.3}
D=\widehat{E}=%
\begin{pmatrix}
-1 & 0  \\ 
0 & -1 
\end{pmatrix}%
\end{equation}
  \begin{equation}\label{C4.4}
F=G=%
\begin{pmatrix}
0 & 0  \\ 
0 & 0 
\end{pmatrix}%
\end{equation}
We introduce the following spaces.
\begin{equation}\label{C4.5}
X=(L^2(0,l))^4, H=X \times \mathbb{R}^2, V=(H^1(0,l))^4, W=V \times \mathbb{R}^2, H_0=\{(u,v,z) \in H, z=0_{\mathbb{R}^2}\},
\end{equation}
We equip these spaces with the following norms:
\begin{equation}\label{C4.6}
\begin{cases}
||(u,v)||_X=\big(\sum_{i=1}^2 (||u_i||_{L^2(0,l)}^2 + ||v_i||_{L^2(0,l)}^2)\big)^{1/2}, \quad \forall \ (u,v)=((u_1,u_2), (v_1,v_2)) \in X, \\
||(u,v,z)||_H=\big(||(u,v)||_X^2 + \sum_{i=1}^2z_i^2\big)^{1/2},\quad \forall \ (u,v,z)=((u_1,u_2), (v_1,v_2), (z_1,z_2)) \in H.
\end{cases}
\end{equation}
Note that $H_0$ is continuously embedded in $H$. It is easy to check that $X$, $H$ and $H_0$ are Hilbert spaces ($H_0$ being a closed subspace of $H$). 
Note also, that if we go back to the use of the Riemann invariants for the reformulation of the Timoshenko system as a coupled first order system, we have 
$$
u=(p,\widehat{\varphi}), v=(q, \psi).
$$
We define the following bounded operator on $H$:
\begin{equation}\label{C4.7}
\Pi: H \longrightarrow H_0, \quad \Pi(u,v,z)^T:=(u,v,0_{\mathbb{R}^2})^T \quad \forall \ (u,v,z) \in H,
\end{equation}
We define the following isometry:
\begin{equation}\label{C4.8}
J: H_0 \longrightarrow X, \quad J(u,v,0_{\mathbb{R}^2})^T=(u,v)^T \quad \forall \ (u,v) \in X,
\end{equation}
and we denote by $J^{-1}$ the inverse operator of $J$. 
We define the following unbounded operator $S_{1,C}$ on $X$:
\begin{equation}\label{C4.9}
D(S_{1,C})=\{(u,v) \in V, u(0)=\widehat{E}v(0), v(l)=Du(l)\}, S_{1,C}=-(\widehat{K}\partial_x +C),
\end{equation}
We define the following unbounded operator $S_{2,C}$ on $H$:
\begin{equation}\label{C4.10}
D(S_{2,C})=\{(u,v,z) \in W, u(0)=\widehat{E}v(0), z=v(l)-Du(l))\}, S_{2,C}=
-\begin{pmatrix}
\widehat{K}\partial_x +C  & 0_{\mathbb{R}^2}\\
0_X & 0_{\mathbb{R}^2}
\end{pmatrix}.
\end{equation}
The authors of [6] recall in Theorem B that $S_{2,C}$ and $S_{1,C_0}$ generate strongly continuous semigroups respectively on $H$ and $X$ (see Theorem B, page 5). Since $S_{1,C}=S_{1,C_0} -(C-C_0)$ and $(C-C_0) \in \mathcal{L}(X)$, we deduce that $S_{1,C}$ is a bounded perturbation of $S_{1,C_0}$, thus it also generates a strongly continuous semigroup on $X$.
We denote by $(e^{tS_{i,C}})_{t \geqslant 0}$ the strongly continuous semigroups generated by the operator $S_{i,C}$ for $i=1,2$.
\end{dfn}
\begin{oss}\label{L2C_0}
Note that the operator presented as the operator $L_2$ in \cite{SW2003} (see page 5) corresponds to the operator denoted here as $S_{1,C_0}$.
\end{oss}

\begin{prop}\label{P1}
Let $S_{1,C}$, and $S_{2,C}$ being respectively defined by \eqref{C4.9} and \eqref{C4.10}. Then we have:
\begin{equation}\label{C4.11}
r_e(e^{tS_{1,C}})=r_e(e^{tS_{2,C}}), \quad \forall \ t \geqslant 0.
\end{equation}
\end{prop}
\begin{proof}
Let us remark that for any $(u_0,v_0) \in D(S_{1,C})$ and any $z_0 \in \mathbb{R}^2$, $\Pi(u_0,v_0, z_0)^T \in D(S_{2,C})$.
Since $D(S_{1,C})$ is dense in $X$, for any $(u_0,v_0) \in X$ there exists a sequence $(u_0^n,v_0^n)_n \subset D(S_{1,C})$ which converges to $(u_0,v_0)$ in $X$. 
Therefore for any  $(u_0,v_0,z_0) \in H$, there exists a sequence $(u_0^n,v_0^n)_n \in D(S_{1,C})$ such that $\Pi(u_0^n,v_0^n, z_0)_n \subset D(S_{2,C})$ and $\Pi(u_0^n,v_0^n, z_0)$ converges to
$\Pi(u_0,v_0, z_0)$ in $H$ as $n$ goes to $\infty$.

Let $(u_0,v_0,z_0) \in D(S_{1,C}) \times \mathbb{R}^2$ be given. We set
\begin{equation}\label{C4.12}
(u_1(t),v_1(t))^T=e^{tS_{1,C}}(u_0,v_0)^T, \quad \forall \ t \geqslant 0,
\end{equation}
and
\begin{equation}\label{C4.13}
(u_2(t),v_2(t), z_2(t))^T=e^{tS_{2,C}}\Pi (u_0,v_0,z_0)^T, \quad \forall \ t \geqslant 0,
\end{equation}
where $\Pi$ is defined in \eqref{C4.7}.
By definition of the operators $S_{i,C}$ for $i=1,2$, this means that $(u_1,v_1)^T$ is the unique solution of
\begin{equation}\label{C4.14}
(u_1^{\prime}(t),v_1^{\prime}(t))^T=-\Big(\widehat{K}( u_{1,x}, v_{1,x})^T + C (u_1,v_1)^T), \quad (u_1(0),v_1(0))^T=(u_0,v_0)^T,
\end{equation}
and that $(u_2,v_2,z_2)^T$ is the unique solution of
\begin{equation}\label{C4.15}
(u_2^{\prime}(t),v_2^{\prime}(t), z_2^{\prime}(t))^T=-\Big(\big(\widehat{K}(u_{2,x}, v_{2,x})^T + C (u_2,v_2)^T\big)^T, 0_{\mathbb{R}^2}\Big)^T, \quad (u_2(0),v_2(0),z_2(0)))^T=(u_0,v_0, 0_{\mathbb{R}^2})^T.
\end{equation}
We note that \eqref{C4.15} is equivalent to
\begin{equation}\label{C4.16}
(u_2^{\prime}(t),v_2^{\prime}(t))^T=-\Big(\widehat{K}( u_{2,x}, v_{2,x})^T + C (u_2,v_2)^T), \quad (u_2(0),v_2(0))^T=(u_0,v_0)^T,
z_2 \equiv 0_{\mathbb{R}^2},
\end{equation}
so that by uniqueness of the solution of \eqref{C4.14}, we have $u_1 \equiv u_2$, $v_1 \equiv v_2$. 

Thus, we proved that for all $(u_0,v_0,z_0) \in D(S_{1,C})\times \mathbb{R}^2$, the following identity holds
\begin{equation}\label{C4.17}
e^{tS_{2,C}}\Pi (u_0,v_0,z_0)^T=(u_1(t),v_1(t), 0_{\mathbb{R}^2})^T= \Big((e^{tS_{1,C}}(u_0,v_0)^T)^T, 0_{\mathbb{R}^2}\Big)^T, \quad \forall \ t \geqslant 0.
\end{equation}
We easily extend this identity for all $(u_0,v_0,z_0) \in H$ by density of $D(S_{1,C})$ in $X$ and the continuity of $\Pi$ and the involved semigroups.
On the other hand, for all $(u_0,v_0,z_0) \in H$, we have
\begin{equation}\label{C4.18}
\Big(J^{-1}e^{tS_{1,C}}J\, \Pi\Big)(u_0,v_0,z_0)^T=\Big( (e^{tS_{1,C}}(u_0,v_0)^T)^T, 0_{\mathbb{R}^2}\Big)^T, \quad \forall \ t \geqslant 0,
\end{equation}
where $J$ is defined in \eqref{C4.8}.
The two above identities lead for all $(u_0,v_0,z_0) \in H$
\begin{multline}\label{C4.19}
\Big(e^{tS_{2,C}}- \Big(J^{-1}e^{tS_{1,C}}J\, \Pi\Big)\Big)(u_0,v_0,z_0)^T=e^{tS_{2,C}}(u_0,v_0,z_0)^T - \Big(J^{-1}e^{tS_{1,C}}\Big)J(u_0,v_0,0)^T\\
=e^{tS_{2,C}}(u_0,v_0,z_0)^T - e^{tS_{2,C}}(u_0,v_0,0)^T= e^{tS_{2,C}}(Id_H-\Pi)(u_0,v_0,z_0)^T, \quad \forall \ t \geqslant 0.
\end{multline}
Since $Id_H - \Pi$ has finite dimensional range, we deduce as in [6] that that $\Big(e^{tS_{2,C}}- \Big(J^{-1}e^{tS_{1,C}}J\, \Pi\Big)\Big)$ is compact, thus we have
\begin{equation}\label{C4.20}
r_e\Big( e^{tS_{2,C}}\Big)= r_e\Big( \Big(J^{-1}e^{tS_{1,C}}J\, \Pi\Big)\Big), \quad \forall \ t \geqslant 0.
\end{equation}
We set $S(t)=\Big(J^{-1}e^{tS_{1,C}}J\Big)$ for all $t \geqslant 0$. Applying Proposition \ref{UptoPi} with $G=H$, $G_0=H_0$, $P=\Pi$, and $R=S(t) \in \mathcal{L}(H_0)$, we get 
\begin{equation}\label{C4.21}
r_e(S(t))= r_e(S(t)\Pi), \quad \forall \ t \geqslant 0.
\end{equation}
Thanks to the two above identities we have
\begin{equation}\label{C4.22}
r_e\Big( e^{tS_{2,C}}\Big)=r_e(S(t)), \quad \forall \ t \geqslant 0.
\end{equation}
It remains to prove that $S(t)$ and $e^{tS_{1,C}}$ have the same essential spectral radius for all $t \geqslant 0$. Let $t \geqslant 0$ be arbitrarily given and set $T(t)=e^{tS_{1,C}}$. We recall that $S(t)=\Big(J^{-1}e^{tS_{1,C}}J\Big)$. For each fixed $t\geqslant 0$, we apply Proposition \ref{Inv} with the choice $G=X$, $F=H_0$, $J : F \mapsto G$ as defined in \eqref{C4.8}, $R=e^{tS_{1,C}}$ and $S=S(t)$. This gives
\begin{equation}\label{C4.31}
r_e(S(t))= r_e\Big(e^{tS_{1,C}}\Big) \quad \forall \ t \geqslant 0.
\end{equation}

Thanks to \eqref{C4.22} and \eqref{C4.31}, we obtain \eqref{C4.11}.
\end{proof}

\subsection{Fulfilling of the second gap in \cite{SW2003} in the proof of non-exponential stability in the case of non equal speed of propagation}

We recall the second gap in the proof of Theorem 2.2 in \cite{SW2003}, namely to define the needed appropriate functional frame and mathematical transformations within this frame that allow to prove that $r_e(e^{tL_1})=r_e(e^{tS_{1,C}})$.

To explain how to fill up the gap, we need to introduce some functional spaces and definitions.

We define the energy space $\mathcal{H}$ associated to \eqref{C3.1}, and a norm on $\mathcal{H}$ as follows.
\begin{multline}\label{C4.32}
\begin{cases}
\mathcal{H}=\big(H^1_0(0,l)\times L^2(0,l)\big)^2, \\
||Y||_{\mathcal{H}}= \Big(||u_x||_{L^2(0,l)}^2 + ||u_2||_{L^2(0,l)}^2 + ||v_x||_{L^2(0,l)}^2 + ||v_2||_{L^2(0,l)}^2\Big)^{1/2},
\forall \ Y=(u,u_2,v,v_2) \in \mathcal{H}.
\end{cases}
\end{multline}
and we define precisely the operator $L_1$ as in \cite{SW2003} as follows
\begin{equation}\label{C4.33}
L_1=\begin{pmatrix}
0 & Id_{H^1_0(0,l)} & 0 & 0 \\
\dfrac{K}{\rho} \partial_{xx} & 0 & - \dfrac{K}{\rho} \partial_{x} & 0 \\
0 & 0 & 0 & Id_{H^1_0(0,l)} \\
\dfrac{K}{I_{\rho} }\partial_{x} & 0 & \dfrac{E\,I}{I_{\rho} }\partial_{xx} & - \dfrac{b}{I_{\rho}}
\end{pmatrix}
\end{equation}
with the domain
\begin{equation}\label{C4.34}
D(L_1)=\Big(\big(H^2(0,l)\cap H^1_0(0,l))\times \big(H^1_0(0,l)\big)\Big)^2.
\end{equation}
\begin{oss}\label{Frame}
In order to explicit, at first formally, the relations between the semigroups $(e^{tL_1})_{t\geqslant 0}$
and $(e^{tS_{1,C}})_{t\geqslant 0}$, one needs to precise the mathematical frame to move from the set of unknowns
$Y(t)=: (u,u_2,v,v_2)^T(t)=: e^{tL_1}Y_0$ to the set of unknowns $\widehat{Y}(t)=:(p,\widehat{\varphi}, q , \psi)^T(t)=: e^{tS_{1,C}}\widehat{Y}_0$ and conversely. 
First note that $u_2=u_t$ and $v_2=v_t$ whenever $Y_0 \in D(L_1)$.
The system \eqref{C3.2} gives the Riemann invariants in terms of the unknowns $(u_x,u_2,v_x,v_2)$. Conversely, we have formally
\begin{equation}\label{C4.35}
\begin{cases}
u_x=\dfrac{1}{2}\sqrt{\dfrac{\rho}{K}}\,\Big(-p+q\Big), \quad u_t= \dfrac{1}{2}(p+q),\\
v_x=\dfrac{1}{2} \sqrt{\dfrac{I_{\rho}}{EI}}\,\Big(-\widehat{\varphi}+\psi\Big), \quad v_t=\dfrac{1}{2}(\widehat{\varphi}+\psi)
\end{cases} 
\end{equation}
Thus if we want to get back to  $Y(t)$ from $\widehat{Y}(t)$, we need to recover from the relations
$$
u(t,x)=\int_0^x u_x(t,y)dy, v(t,x)=\int_0^x v_x(t,y)dy, t \geqslant 0, x \in (0,l)
$$
the properties $u(t,l)=0$, and $v(t,l)=0$ for $t \geqslant 0$. This means that the properties
$$
\int_0^l u_x(t,y)dy=0,  \int_0^l v_x(t,y)dy=0, t \geqslant 0
$$
should hold. Going back to the Riemann invariants \eqref{C4.35}, this requires that $(p,\widehat{\varphi}, q , \psi)$ satisfies
$$
\int_0^l (p-q)dx=0, \int_0^l (\widehat{\varphi}-\psi)dx=0.
$$
Hence we have to work with the restriction of the operator $S_{1,C}$ on the subspace of vector functions $(p,\widehat{\varphi}, q , \psi) \in X$ such that $\int_0^l (p-q)dx=0, \int_0^l (\widehat{\varphi}-\psi)dx=0$ and to prove several properties before proving that $r_e(e^{tL_1})=r_e(e^{tS_{1,C}})$.
\end{oss}

From now on, we identify without further details, any $U$ in $X$ with its transpose $U^T$ whenever necessary without further details. We need to introduce the functional requested frame. For this, we set the following definitions and notations:
\begin{dfn}
We define on $X$ the linear forms $R_i$ for $i=1,2$ by
\begin{equation}\label{C4.37}
R_1U^T=\int_0^l(p-q)(x)dx, R_2U^T=\int_0^l(\widehat{\varphi}-\psi)(x)dx, \quad \forall \ U=(p, \widehat{\varphi}, q, \psi) \in X,
\end{equation}
and
\begin{equation}\label{C4.38}
\begin{cases}
X_0=\mathcal{N}(R_1) \cap \mathcal{N}(R_2) \subset X, \\
\Pi_0: X \mapsto X_0 \mbox{ the orthogonal projection of } X \mbox{ on } X_0,

\end{cases}
\end{equation}
where $\mathcal{N}(T)$ stands for the null space of any continuous linear form $T \in \mathcal{L}(X,\mathbb{R})$.
We define the two following linearly independent elements $e_1 \in X$ and $e_2 \in X$ 
\begin{equation}\label{C4.39}
e_1=\dfrac{1}{\sqrt{2l}}(1, 0, -1, 0)^T , e_2=\dfrac{1}{\sqrt{2l}}(0,1,0,-1)^T, E_1= \mathbb{R}e_1 \bigoplus \mathbb{R}e_2.
\end{equation}
\end{dfn}
It is easy to check that the following properties hold:
\begin{prop}\label{Prop3}
Let the constant vectors $e_1$, $e_2$, the spaces, $X$, $E_1$ and $X_0$ being defined as above, then we have
\begin{equation}\label{C4.40}
\begin{cases}
(i) \quad e_1\perp e_2\,, ||e_1||_X=||e_2||_X=1, \\
(ii) \quad \quad E_1= X_0^{\perp}\,, X= X_0 \bigoplus E_1,\\
(iii) \quad \Pi_0U = \displaystyle{U -\dfrac{1}{2l}\sum_{i=1}^2\langle U,e_i\rangle_X \;e_i }\quad \forall \ U \in X.
\end{cases}
\end{equation}

\end{prop}
\begin{dfn}
For any $L \in \mathcal{L}(X)$, we denote by $L_{| X_0}$ the restriction of the linear operator $L$ on $X_0$.
\end{dfn}
We now prove the following results.
\begin{lem}\label{Lm1}
Let $S_{1,C}$ being defined by \eqref{C4.9}.
We have the following properties
\begin{equation}\label{C4.41}
\begin{cases}
(i) \quad S_{1,C}\big(D(S_{1,C})\big) \subset X_0, \\
(ii) \quad D(S_{1,C}) \cap X_0 \mbox{ is dense in } X_0,\\
(iii) \quad e^{tS_{1,C}}W_0 \in X_0\,, \quad \forall \ t \geqslant 0\,, \quad \forall \ W_0 \in X_0,\\
(iv) \quad \forall t \geqslant 0, e^{tS_{1,C}}|_{X_0} \in \mathcal{L}(X_0).
\end{cases}
\end{equation}

\end{lem}
\begin{proof}
We shall first prove that for $i=1,2$, we have
$$
R_iS_{1,C}W_0=0, \forall \ W_0 \in D(S_{1,C}).
$$
Let $W_0=(p_0,\widehat{\varphi}_0,q_0,\psi_0)^T \in D(S_{1,C})$ be given arbitrarily. Then we have
\begin{multline}\label{C4.42}
S_{1,C}W_0=\Big(-\sqrt{K/\rho}\; p_{0,x} + (K/(2\rho)) \sqrt{I_{\rho}/ E\, I}\; (\widehat{\varphi}_0-\psi_0), \\
-\sqrt{E\, I/I_{\rho}}\; \widehat{\varphi}_{0,x}
-(\sqrt{\rho K}/(2I_{\rho}))\;(p_0-q_0) - (b/(2I_{\rho}))\;(\widehat{\varphi}_0+\psi_0), 
\sqrt{K/\rho}\; q_{0,x} + (K/(2\rho))\sqrt{I_{\rho}/ E\, I}\; (\widehat{\varphi}_0-\psi_0), \\
\sqrt{E\, I/I_{\rho}}\; \psi_{0,x}
-(\sqrt{\rho K}/(2I_{\rho}))\;(p_0-q_0) - (b/(2I_{\rho}))\;(\widehat{\varphi}_0+\psi_0)
\Big)^T,
\end{multline}
so that since $W_0 \in D(S_{1,C})$, we have
$$
R_1S_{1,C}W_0=-\sqrt{K/\rho}\int_0^l(p_0+q_0)_xdx= -\sqrt{K/\rho}\big((p_0+q_0)(l)-(p_0+q_0)(0)\big)=0,
$$
and
$$
R_2S_{1,C}W_0=-\sqrt{E\,I/\rho}\int_0^l(\widehat{\varphi}_0+\psi_0)_xdx= -\sqrt{E\,I/\rho}\big((\widehat{\varphi}_0+\psi_0)(l)-(\widehat{\varphi}_0+\psi_0)(0)\big)=0.
$$
This proves the first property $(i)$ of \eqref{C4.41}. 
Let us now prove the property $(ii)$. Let $W_0=:(u_0, v_0)^T \in X_0$ be given arbitrarily. Since $X_0$ is continuously embedded in $X$ and $D(S_{1,C})$ is dense in $X$, there exists a sequence $(W_{0,n})_n \subset D(S_{1,C})$, such that $(W_{0,n})_n$ converges to $W_0$ as $n$ goes to infinity. Recalling that the orthogonal projection $\Pi_0$ of $X$ onto $X_0$ defined in \eqref{C4.38} is continuous, and since $W_0 \in X_0$, we deduce that $(\Pi_0W_{0,n})_n$ converges to $W_0$ in $X_0$. 
Let us now check that $\Pi_0 W_{0,n} \in D(S_{1,C})$ for all $n \in \mathbb{N}$. 
We have $W_{0,n}=:(u_{0,n},v_{0,n})^T
\in D(S_{1,C})$ for all $n \in \mathbb{N}$, which is equivalent to the properties that for all integers $n$, $(u_{0,n},v_{0,n}) \in (H^1(0,l))^4$ and $u_{0,n}(0)=\widehat{E}v_{0,n}(0)$ and $v_{0,n}(l)=Du_{0,n}(l)$. We now remark that the vectors $e_1$ and $e_2$ defined in \eqref{C4.39} are constant vector functions. With the block notation we used for the definition of the operator $S_{1,C}$, we can set $e_i=(u_{e_i}, v_{e_i})^T$ for $i=1,2$ where $u_{e_1}=1/\sqrt{2l}(1,0)$, $v_{e_1}=1/\sqrt{2l}(-1,0)$, $u_{e_2}=1/\sqrt{2l}(0,1)$, $v_{e_2}=1/\sqrt{2l}(0, -1)$, so that we remark that $u_{e_{i}}(0)=\widehat{E}v_{e_i}(0)$ and $v_{e_{i}}(l)=Du_{e_{i}}(l)$. for $i=1,2$. This implies that $e_i \in D(S_{1,C})$ for $i=1,2$. Thus we have
$$
\displaystyle{\dfrac{1}{2l}\sum_{i=1}^2\langle W_{0,n},e_i\rangle_X \;e_i  \in D(S_{1,C})}.
$$
On the other hand, thanks to the explicit expression of $\Pi_0$ given in \eqref{C4.40}-$(iii)$ (see Proposition \ref{Prop3}), we have
$$
\Pi_0 W_{0,n}=W_{0,n}- \displaystyle{\dfrac{1}{2l}\sum_{i=1}^2\langle W_{0,n},e_i\rangle_X \;e_i 
}.
$$
This together, with the above property and the property that $W_{0,n} \in D(S_{1,C})$, imply that $\Pi_0 W_{0,n} \in D(S_{1,C})$ for all $n \in \mathbb{N}$. Since we proved that $(\Pi_0W_{0,n})_n$ converges to $W_0$ in $X_0$, we deduce the desired property $(ii)$.

Let us now prove the property $(iii)$.
We set
$$
W(t)=(p,\widehat{\varphi},q,\psi)^T(t)=e^{tS_{1,C}}W_0\,, \quad \forall \ t \geqslant 0, \quad \ W_0=(p_0,\widehat{\varphi}_0,q_0,\psi_0)^T \in D(S_{1,C})\cap X_0.
$$
Then, we have
$$
(p_t,\widehat{\varphi}_t,q_t,\psi_t)^T=S_{1,C}W(t) \,, \quad \forall \ t \geqslant 0.
$$
Using the property $(i)$ of \eqref{C4.41}, we already proved, we deduce that
$$
R_1(p_t,\widehat{\varphi}_t,q_t,\psi_t)^T=\int_0^l(p-q)_t(t,x)dx=0, \quad R_2(p_t,\widehat{\varphi}_t,q_t,\psi_t)^T=\int_0^l(\widehat{\varphi}-\psi)_t(t,x)dx=0.
$$
Therefore, since $W_0 \in D(S_{1,C})\cap X_0$, we have
$$
\int_0^l(p-q)(t,x)dx=\int_0^l(p_0-q_0)(x)dx=0, \quad \int_0^l(\widehat{\varphi}-\psi)(t,x)dx=\int_0^l(\widehat{\varphi}_0-\psi_0)(x)dx=0,
\forall \ t \geqslant 0,$$
so that $W(t) \in X_0$ for all $t \geqslant 0$. 
Hence 
$$
e^{tS_{1,C}}W_0 \in X_0\,, \quad \forall \ t \geqslant 0\,, \quad \forall \ W_0 \in D(S_{1,C})\cap X_0.
$$
Since $D(S_{1,C})\cap X_0 $ is dense in $X$ by the property $(ii)$ we just proved, we deduce that
$$
e^{tS_{1,C}}W_0 \in X_0\,, \quad \forall \ t \geqslant 0\,, \quad \forall \ W_0 \in X_0,
$$
and $(e^{tS_{1,C}}|_{X_0})_{t\geqslant 0}$ is well-defined as a family of linear bounded operators acting in $X_0$. This concludes the proof.
\end{proof}

\begin{oss}\label{S1C_0notX_0stable}
Note that $X_0$ is not stable by the semigroup generated on $X$ by $S_{1,C_0}$, contrarily to the one generated by $S_{1,C}$ as proved in Lemma \ref{Lm1}.
\end{oss}

\begin{lem}\label{Lm2}
Let $S_{1,C}$ being defined by \eqref{C4.9}. We consider the restriction to $X_0$ of the unbounded operator $S_{1,C}$, that we denote by $S_{1,C}\small{|}_{X_0}$. Then we have the following properties:

$S_{1,C}\small{|}_{X_0}$ is well-defined as an unbounded operator on $X_0$ with domain $D(S_{1,C}\small{|}_{X_0})$ where 
\begin{equation}\label{C4.43}
(i) \quad D(S_{1,C}\small{|}_{X_0})=D(S_{1,C})\cap X_0. 
\end{equation}
Moreover $S_{1,C}\small{|}_{X_0}$ generates a strongly continuous semigroup $(e^{tS_{1,C}\small{|}_{X_0}})_{t \geqslant 0}$ on $X_0$ which coincides 
with the restriction to $X_0$ of the semigroup generated by $S_{1,C}$, that is:
$$
(ii) \quad e^{tS_{1,C}\small{|}_{X_0}}=e^{tS_{1,C}}|_{X_0} \forall \ t \geqslant 0.
$$
Moreover we have
\begin{equation}\label{C4.44}
(iii) \quad r_e(e^{tS_{1,C}})=r_e(e^{tS_{1,C}\small{|}_{X_0}}), \quad \forall \ t \geqslant 0.
\end{equation}
\end{lem}
\begin{proof}
In order to define rigorously the operator $S_{1,C}\small{|}_{X_0}$, we need to define its domain. The domain $D(S_{1,C}\small{|}_{X_0})$ is defined as the set $\{W^0 \in X_0, S_{1,C}^0W_0 \in X_0\}$, so that $D(S_{1,C}\small{|}_{X_0})=\{W_0 \in X_0, W_0 \in D(S_{1,C}), S_{1,C}W_0 \in X_0\}$. However thanks to property $(i)$ of Lemma \ref{Lm1}, we have $S_{1,C}W_0 \in X_0$ for all $W_0 \in D(S_{1,C})$. Hence we have
$D(S_{1,C}\small{|}_{X_0})=D(S_{1,C})\cap X_0$ and $S_{1,C}\small{|}_{X_0}$ generates a strongly continuous semigroup $(e^{tS_{1,C}\small{|}_{X_0}})_{t \geqslant 0}$ on $X_0$. Let us now prove that the property $(ii)$ holds. We easily prove that
for all $t \geqslant 0$, $e^{tS_{1,C}\small{|}_{X_0}}$ and $e^{tS_{1,C}}\small{|}_{X_0}$ coincide on $D(S_{1,C})$, which is dense in $X_0$ (see property $ii)$ of \ref{Lm1}). This together with the continuity of $e^{tS_{1,C}\small{|}_{X_0}}$ and $e^{tS_{1,C}}\small{|}_{X_0}$ over $X_0$ lead to property $(ii)$.

 Let us now prove $(iii)$. We first remark that we have 
 \begin{equation}\label{C4.45}
 e^{tS_{1,C}}=e^{tS_{1,C}}\Pi_0 + e^{tS_{1,C}}(Id_X - \Pi_0), \forall \ t \geqslant 0.
 \end{equation}
 Since $(Id_X - \Pi_0)$ is a compact operator in $X$ (it takes values in the two dimensional subspace $E_1$) and $e^{tS_{1,C} }\in \mathcal{L}(X)$, $e^{tS_{1,C}}(Id_X - \Pi_0)$ is a compact operator in $X$. Thus we have:
 \begin{equation}\label{C4.46}
 r_e(e^{tS_{1,C}})=r_e(e^{tS_{1,C}}\Pi_0), \forall \ t \geqslant 0.
 \end{equation}
 
 Applying Proposition \ref{UptoPi} with $G=X$, $G_0=X_0$, $P=\Pi_0$, and $R= e^{tS_{1,C}\small{|}_{X_0}} \in \mathcal{L}(X_0)$, we get 
 \begin{equation}\label{C4.47}
 r_e(e^{tS_{1,C}\small{|}_{X_0}})=r_e(e^{tS_{1,C}\small{|}_{X_0}}\Pi_0), \forall \ t \geqslant 0.
 \end{equation}
 On the other hand, we have
 \begin{equation}\label{C4.48}
 e^{tS_{1,C}}\,\Pi_0=e^{tS_{1,C}}\small{|}_{X_0}\, \Pi_0=e^{tS_{1,C}\small{|}_{X_0}}\,\Pi_0, \forall \ t \geqslant 0.
 \end{equation}
 This implies
 \begin{equation}\label{C4.49}
 r_e(e^{tS_{1,C}}\Pi_0)= r_e(e^{tS_{1,C}\small{|}_{X_0}}\,\Pi_0), \forall \ t \geqslant 0.
 \end{equation}

 Using successively \eqref{C4.46}, \eqref{C4.49} and \eqref{C4.47}, we deduce that \eqref{C4.44} holds, that is 
$r_e(e^{tS_{1,C}})=r_e(e^{tS_{1,C}\small{|}_{X_0}})$ for all  $t \geqslant 0$, so that property $(iii)$ is proved.

 \end{proof}

Before to complete the filling of the second gap in the proof of Theorem 2.2 in \cite{SW2003} , we need to define some operators and to recall some spaces and norms we already introduced.

\begin{prop}\label{R}
Let $\mathcal{H}$ and $||\cdot||_{\mathcal{H}}$ be defined as in \eqref{C4.32}. Let $X$ be the space defined in \eqref{C4.5}, equipped with the natural norm defined next to \eqref{C4.5}. We set
\begin{equation}\label{C4.36}
L^2_0(0,l)=\Big\{f \in L^2(0,l), \int_0^l f(x)dx=0\Big\}.
\end{equation}
We define the subspace $\widehat{X}=\big(L_0^2(0,l))\times L^2(0,l)\big)^2 \subset X$, equipped with the induced norm $||\cdot||_X$.
We define the operator $R: \mathcal{H} \mapsto \widehat{X}$ by
\begin{equation}\label{C4.51R}
RY=(u_x,u_2,v_x,v_2)^T, \quad \forall \ Y=(u,u_2,v,v_2)^T \in \mathcal{H}.
\end{equation}
Then $R$ is an isometry from $\mathcal{H}$ onto $\widehat{X}$. 
\end{prop}
\begin{proof}
Let $Y=(u,u_2,v,v_2)^T \in \mathcal{H}$ be given arbitrarily. Then it is easy to check that
$||RY||_X=||Y||_{\mathcal{H}}$. Thus the assumption $||RY||_X=0$ implies $Y=0$. Thus $R$ is one-to-one. Let
$U=(f,g,h,r)^T$ be given arbitrarily in  $\widehat{X}$. We set 
\begin{equation}\label{C4.52}
u(x)=\int_0^xf(s)ds \;  \forall \ x \in (0,l), v(x)=\int_0^xh(s)ds \; \forall \ x \in (0,l).
\end{equation}
Then since $f\in L^2_0(0,l)$ (resp. $h\in L^2_0(0,l)$), $u \in H^1_0(0,l)$ (resp. $v \in H^1_0(0,l)$). Hence $Y=(u,g,v,r)^T \in \mathcal{H}$ and $RY=U$. Thus $R$ is invertible and its inverse is defined by $R^{-1}U=Y$ with $Y$ defined from $U$ as above. This concludes the proof.
\end{proof}
The way to derive the Riemann invariants $W=(p,\widehat{\varphi}, q, \psi)^T$ given formally by \eqref{C3.2} from $U=(f,g,h,r)^T$ and conversely is described in the next proposition, which proof is left to the reader.
\begin{prop}\label{M}
Let $X$ be defined as in \eqref{C4.5}, and $X_0 \subset X$ be defined in \eqref{C4.38}.
We define the operator $\widehat{M}: \widehat{X} \mapsto X$
by
\begin{equation}\label{C4.53M}
\begin{cases}
\widehat{M}U=W \quad \forall \ U=(f,g,h,r)^T \in \widehat{X} , \mbox{ where }\\
W=(p,\widehat{\varphi}, q, \psi)^T=\Big(-\sqrt{\dfrac{K}{\rho}}f + g, -\sqrt{\dfrac{EI}{I_{\rho}}}h + r, \sqrt{\dfrac{K}{\rho}}f + g,
\sqrt{\dfrac{EI}{I_{\rho}}}h + r\Big)^T.
\end{cases}
\end{equation}

Then $\widehat{M}$ is an isomorphism from $\widehat{X}$ onto $X_0$ and the inverse operator $M^{-1}: X_0 \mapsto \widehat{X}$
is given by
\begin{equation}\label{C4.54}
\begin{cases}
\widehat{M}^{-1}W= U \quad \forall \ W=(p,\widehat{\varphi}, q, \psi)^T \in X_0 , \mbox{ where }\\
U=(f,g, h, r)^T=\Big(\dfrac{1}{2}\sqrt{\dfrac{\rho}{K}}\,\Big(-p+q\Big), \dfrac{1}{2}(p+q), \dfrac{1}{2} \sqrt{\dfrac{I_{\rho}}{EI}}\,\Big(-\widehat{\varphi}+\psi\Big), \dfrac{1}{2}(\widehat{\varphi}+\psi)\Big)^T
\end{cases}
\end{equation}
\end{prop}
We introduced the needed functional frame and gave the idea in Remark \ref{Frame} of which infinitesimal generator, namely the restriction to $X_0$ of $S_{1,C}$, is needed to be compared to the infinitesimal generator $L_1$. In the next Proposition, we show that the requested operator to describe the links between these two infinitesimal generators is the composed operator $\widehat{M}R$.
\begin{prop}\label{RMI}
Let $R$ and $\widehat{M}$ be respectively defined by \eqref{C4.51R} and \eqref{C4.53M}. Then the composed operator $\widehat{M}R$ satisfies the following properties
\begin{equation}\label{4.55}
\begin{cases}
\displaystyle{(\widehat{M}R) \mbox{ is the isomorphism from}: \mathcal{H} \mbox{ onto } X_0,\mbox{ defined by:}}\\
\displaystyle{(\widehat{M}R)Y=(p,\widehat{\varphi}, q, \psi)^T, \quad \forall \ Y=(u,u_2,v,v_2)^T \in \mathcal{H} \mbox{ where:}}\\
\displaystyle{p=-\sqrt{\dfrac{K}{\rho}}u_x + u_2, \widehat{\varphi}=-\sqrt{\dfrac{EI}{I_{\rho}}}v_x + v_2},\\
\displaystyle{q=\sqrt{\dfrac{K}{\rho}}u_x + u_2,
\psi=\sqrt{\dfrac{EI}{I_{\rho}}}v_x + v_2 ,}\\
\end{cases}
\end{equation}
\begin{equation}\label{4.56}
\begin{cases}
\displaystyle{(\widehat{M}R)^{-1}: X_0 \mapsto \mathcal{H} \mbox{ is given by:}}\\
\displaystyle{(\widehat{M}R)^{-1}W=(u,u_2,v,v_2)^T, \quad \forall \ W=(p,\widehat{\varphi}, q, \psi)^T \in X_0 \mbox{ where:}}\\
\displaystyle{u(\cdot)=\dfrac{1}{2}\sqrt{\dfrac{\rho}{K}}\int_0^{\cdot}(-p+q)(s)ds, u_2= \dfrac{1}{2}(p+q)},\\
\displaystyle{v(\cdot)=\dfrac{1}{2} \sqrt{\dfrac{I_{\rho}}{EI}}\int_0^{\cdot}(-\widehat{\varphi}+\psi)(s)ds, v_2=\dfrac{1}{2}(\widehat{\varphi}+\psi)},
\end{cases}
\end{equation}
and the following characterization holds:
\begin{equation}\label{4.57}
(\widehat{M}R)(D(L_1))=D(S_{1,C})\cap X_0.
\end{equation}
\begin{proof}
The proof of \eqref{4.55} and \eqref{4.56} can easily be deduced from Proposition \ref{R} and Proposition \ref{M}. Let us prove \eqref{4.57}. Let $Y=(u,u_2,v,v_2)^T$ be given arbitrarily in $D(L_1)$. This is equivalent to assume that $(u,v) \in (H^2(0,l)\cap H^1_0(0,l))^2$ and $(u_2,v_2) \in (H^1_0(0,l))^2$. We set $(\widehat{M}R)Y=W=(p,\widehat{\varphi}, q, \psi)^T$. Then the components of $W$ are given in \eqref{4.55}. Since $u_x$ and $u_2$ are in $H^1(0,l)$, we deduce that $p$ and $q$ are $H^1(0,l)$. Similarly, since $v_x$ and $v_2$ are in $H^1(0,l)$, we deduce that $\widehat{\varphi}$ and $\psi$ are in $H^1(0,l)$. Moreover thanks once again to \eqref{4.55}, we have $(p+q)=2u_2$ and $\widehat{\varphi} + \psi=2v_2$. This, together with $u_2 \in H^1_0(0,l)$ and $v_2 \in H^1_0(0,l)$, imply that $(p+q)(x)=(\widehat{\varphi} + \psi)(x)=0$ for $x=0$ and $x=l$. This proves that $W \in D(S_{1,C})$. On the other hand, we already know that $\widehat{M}\widehat{X}=X_0$. Thus $W \in D(S_{1,C})\cap X_0$. This proves that $(\widehat{M}R)(D(L_1))\subset D(S_{1,C})\cap X_0$. 

Let us now prove the reverse inclusion. For this, let 
$W=(p,\widehat{\varphi}, q, \psi)^T$ be given arbitrarily in $D(S_{1,C})\cap X_0$. This is equivalent $W \in (H^1(0,l))^4$, 
$(p+q)(x)=(\widehat{\varphi} + \psi)(x)=0$ for $x=0$ and $x=l$ and $\int_0^l(p-q)ds=\int_0^l(\widehat{\varphi}-\psi)ds=0$. We set $(\widehat{M}R)^{-1}W=Y=(u,u_2,v,v_2)^T$. Then the components of $Y$ are given in \eqref{4.56}. Since $p$ and $q$ are in $H^1(0,l)$, we deduce easily that $u \in H^2(0,l)$. Moreover since $\int_0^l(p-q)ds=0$, we also have $u \in H^1_0(0,l)$. In a similar way, we  prove that $v \in H^2(0,l) \cap H^1_0(0,l)$. Using \eqref{4.56} once again and since $(p+q)(0)=(p+q)(l)=0$, we deduce that $u_2 \in H^1_0(0,l)$. We prove that $v_2 \in H^1_0(0,l)$ in a similar way. This proves that $Y \in D(L_1)$. Hence 
$(\widehat{M}R)^{-1}(D(S_{1,C})\cap X_0) \subset D(L_1)$. Therefore $D(S_{1,C})\cap X_0 \subset (\widehat{M}R)(D(L_1))$.
\end{proof}
\end{prop}

In the next Corollary, we describe precisely the mathematical transformations which make it possible to pass from the semigroup generated by the unbounded operator $L_1$ on $\mathcal{H}$ to the semigroup generated by the restriction to $X_0$ of $S_{1,C}$, and conversely.

\begin{cor}\label{Cor3}
Let $L_1$ be the operator defined in \eqref{C4.33}-\eqref{C4.34} and the operator $S_{1,C}$ defined in \eqref{C4.9}, and the operators $R$ and $\widehat{M}$ be respectively defined by \eqref{C4.51R} and \eqref{C4.53M}. Then we have the following
identity
\begin{equation}\label{4.58}
\widehat{M}R\; e^{tL_1}\; R^{-1}\widehat{M}^{-1}=e^{tS_{1,C}\small{|}_{X_0}}, \quad \forall  t \geqslant 0.
\end{equation}
\end{cor}
\begin{proof}
Let $W_0=(p_0,\widehat{\varphi}_0, q_0, \psi_0)^T$ be given arbitrarily in $D(S_{1,C})\cap X_0$. We set 
\begin{equation}\label{4.59}
Y_0=(\widehat{M}R)^{-1}W_0=(u_0,u_{2,0},v_0,v_{2,0})^T.
\end{equation}
Thanks to \eqref{4.57} in Proposition \ref{RMI}, we have $Y_0 \in D(L_1)$. We define for all $t\geqslant 0$: 

$Y(t)=e^{tL_1}Y_0=:(u,u_2,v,v_2)^T(t)$. Since  $Y_0 \in D(L_1)$ $Y \in \mathcal{C}([0,\infty); D(L_1)) \cap \mathcal{C}^1([0,\infty); \mathcal{H})$ and solves the Cauchy problem
$$
Y'(t)=L_1Y(t) \quad \forall \ t \geqslant 0, Y(0)=Y_0,
$$
which is equivalent to
\begin{equation}\label{4.60}
\begin{cases}
u_2=u_t, v_2=v_t \quad \forall \ t \geqslant 0, \forall \ x \in (0,l),\\
u_{tt}=\dfrac{K}{\rho}(u_{x}-v)_{x},  \quad t\geqslant 0, x \in (0,l), \\[1em]
v_{tt}=\dfrac{EI}{I_{\rho }}v_{xx}+\dfrac{K}{I_{\rho }}u_{x}-\dfrac{b}{I_{\rho }}v_{t}, \quad t\geqslant 0, x \in (0,l), \\
u(t,0)=u(t,l)=0,\quad v(t,0)=v(t,l)=0, \quad t \geqslant 0,\\
u(0,\cdot)=u_0(\cdot),u_t(0,\cdot)=u_{2,0}(\cdot), v(0,\cdot)=v_0(\cdot),v_t(0,\cdot)=v_{2,0}(\cdot), \quad x \in (0,l).
\end{cases}
\end{equation}
We now set $W(\cdot)=:\widehat{M}RY(\cdot)=:(p,\widehat{\varphi}, q, \psi)^T(\cdot)$. By the definition \eqref{4.59} of $Y_0$, we have
\begin{equation}\label{4.61}
W(0)=W_0
\end{equation}
Thanks to \eqref{4.57}, $W(t) \in D(S_{1,C} \cap X_0$ for all $t \geqslant 0$, and thanks to Proposition \ref{RMI}, the components of $W$ are given by the last two equations in \eqref{4.55}, that is
\begin{equation}\label{4.62}
\begin{cases}
\displaystyle{p=-\sqrt{\dfrac{K}{\rho}}u_x + u_t, \widehat{\varphi}=-\sqrt{\dfrac{EI}{I_{\rho}}}v_x + v_t},\\
\displaystyle{q=\sqrt{\dfrac{K}{\rho}}u_x + u_t,
\psi=\sqrt{\dfrac{EI}{I_{\rho}}}v_x + v_t,}\\
\end{cases}
\end{equation}
We just now just need to check that $W$ solves the Cauchy problem
$$
W'(t)=S_{1,C}\small{|}_{X_0}W(t) \quad \forall \ t \geqslant 0, W(0)=W_0.
$$
which is equivalent to check that $(p,\widehat{\varphi}, q, \psi)$ solves the system
\begin{equation}\label{4.63}
\begin{cases}
p_t=-\sqrt{K/\rho}\; p_{x} + (K/(2\rho)) \sqrt{I_{\rho}/ E\, I}\; (\widehat{\varphi}-\psi), \\
\widehat{\varphi}_t=-\sqrt{E\, I/I_{\rho}}\; \widehat{\varphi}_{x}
-(\sqrt{\rho K}/(2I_{\rho}))\;(p-q) - (b/(2I_{\rho}))\;(\widehat{\varphi}+\psi), \\
q_t=\sqrt{K/\rho}\; q_{x} + (K/(2\rho))\sqrt{I_{\rho}/ E\, I}\; (\widehat{\varphi}-\psi), \\
\psi_t=\sqrt{E\, I/I_{\rho}}\; \psi_{x}
-(\sqrt{\rho K}/(2I_{\rho}))\;(p-q) - (b/(2I_{\rho}))\;(\widehat{\varphi}+\psi),\\
W(0)=W_0.
\end{cases}
\end{equation}
We get easily from \eqref{4.62} that 
\begin{equation}\label{4.64}
(\widehat{\varphi}-\psi)= -2\sqrt{E\, I / I_{\rho}}\; v_x, \ p-q= -2 \sqrt{K/\rho}\; u_x, \ (\widehat{\varphi}+\psi)=2v_t, p+q=2u_t.
\end{equation}
Differentiating with respect to $t$ the expression of $p$ given in the first line of \eqref{4.62}, we get
$$
p_t=-\sqrt{\dfrac{K}{\rho}}u_{xt}+ u_{tt} 
$$
Similarly, differentiating with respect to $x$ the expression of $p$ given in the first line of \eqref{4.62}, and multiplying
the resulting identity by $\sqrt{\dfrac{K}{\rho}}$, we get
$$
(\sqrt{K/\rho})\; p_{x}=-(K/\rho)\; u_{xx} + \sqrt{K/\rho}\; u_{tx}
$$
Using the two above identities together with the second line of \eqref{4.60} (which gives the system satisfied by $(u,v)$), we get
$$
p_t+\sqrt{K/\rho}\; p_{x}=u_{tt}-(K/\rho)\; u_{xx}=-(K/\rho)\;v_x
$$
Using \eqref{4.64} in the above identity, we obtain 
$$
p_t+\sqrt{K/\rho}\; p_{x}  - (K/(2\rho)) \sqrt{I_{\rho}/ E\, I}\; (\widehat{\varphi}-\psi)=0
$$
This proves the first identity of \eqref{4.63}.
 We proceed in a similar way to prove that
$$
\widehat{\varphi}_t+\sqrt{E\, I/I_{\rho}}\; \widehat{\varphi}_{x}=v_{tt} - E\, I/I_{\rho}v_{xx}=\dfrac{K}{I_{\rho }}u_{x}-\dfrac{b}{I_{\rho }}v_{t}
$$
where we used the third equation of \eqref{4.60} for the last equality.
Using \eqref{4.64}, we deduce from the above identity, the second identity of \eqref{4.63}. We proceed in a similar way to prove the third and fourth identity of \eqref{4.63}. 
Thus we proved that 
$$
\widehat{M}R\; e^{tL_1}\; R^{-1}\widehat{M}^{-1}W_0=e^{tS_{1,C}\small{|}_{X_0}}W_0, \quad \forall  t \geqslant 0,\; \forall \ W_0 \in D(S_{1,C})\cap X_0.
$$
Thanks to Lemma \ref{Lm1}-$(ii)$, $D(S_{1,C})\cap X_0$ is dense in $X_0$. This together with the above identity proves \eqref{4.58}. This concludes the proof.
\end{proof}

\begin{prop}\label{P2}
We consider the operator $L_1$ defined in \eqref{C4.33}-\eqref{C4.34} and the operator $S_{1,C}$ defined in \eqref{C4.9}. Then we have the following property:
\begin{equation}\label{C4.50}
r_e(e^{tL_1})=r_e(e^{tS_{1,C}}), \quad \forall \ t \geqslant 0.
\end{equation}
\end{prop}
\begin{proof}
Thanks to Corollary \ref{Cor3} and to Proposition \ref{Inv}, we have $r_e(e^{tL_1})=r_e(e^{tS_{1,C}\small{|}_{X_0}})$ for all $t \geqslant 0$. This together with the property $(iii)$ in Lemma \ref{Lm2} allows us to conclude the proof.
\end{proof}
Thanks to the new results we stated above, we can now prove the following result that gives the complete proof of the identity between the essential radius of the strongly continuous semigroup generated by $L_1$ and the essential radius of the strongly continuous semigroup generated by $S_{1,C_0}$:
\begin{prop}\label{P2C0}
We consider the operator $L_1$ defined in \eqref{C4.33}-\eqref{C4.34} and the operator $S_{1,C_0}$ defined in \eqref{C4.9} with $C_0$ (defined in \eqref{C0}) replacing $C$. Then we have the following property:
\begin{equation}\label{C4.51}
r_e(e^{tL_1})=r_e(e^{tS_{1,C_0}}), \quad \forall \ t \geqslant 0.
\end{equation}
\end{prop}
\begin{proof}
Thanks to Proposition \ref{P2} and to Proposition \ref{P1}, we respectively have $r_e(e^{tL_1})=r_e(e^{tS_{1,C}})$  and $r_e(e^{tS_{1,C}})=r_e(e^{tS_{2,C}})$ for all $t \geqslant 0$. Hence we have
$r_e(e^{tL_1})=r_e(e^{tS_{2,C}})$ for all $t \geqslant 0$. Thanks to Theorem A in \cite{NRL1986}, and to the requested result of Proposition \ref{UptoPi} for a complete proof of Theorem A in
\cite{NRL1986}, we have $r_e(e^{tS_{2,C}})=r_e(e^{tS_{1,C_0}})$.  The two last identities  allow us to conclude the proof.
\end{proof}
\begin{oss}\label{ThmA}
Note that the arguments of the proof of Theorem A in \cite{NRL1986} stated at the pages 334-335 allow only to prove that 
$$
r_e(e^{tS_{2,C}})=r_e(e^{tS_{1,C_0}}\Pi) \quad \forall \ t \geqslant 0,
$$
where $\Pi$ is defined in \eqref{C4.7}. The proof that 
$$
r_e(e^{tS_{1,C_0}}\Pi) = r_e(e^{tS_{1,C_0}}) \quad \forall \ t \geqslant 0
$$
is not given. 

This last identity results from an application of our Proposition \ref{UptoPi}. For understanding why the general property proved in Proposition \ref{UptoPi} is not directly resulting from the property of invariance of the essential radius of bounded operators through the action of additive compact operators, one should note that the two operators $e^{tS_{2,C}}$ and $e^{tS_{1,C_0}}$ are acting on different Hilbert spaces, which are respectively the Hilbert space $H=(L^2(0,l))^4 \times \mathbb{R}^2$ and the Hilbert space $(L^2(0,l))^4$, so that they cannot differ by the action of an additive compact operator. An invariance property of the essential spectral radius by the right multiplicative action by the orthogonal projection operator on a suitable subspace has to be proved. This is performed in our Proposition \ref{UptoPi} (see also Proposition \ref{UptoPiB} for a more general formulation in Banach spaces). 
\end{oss}

\begin{oss}\label{L1L2}
Note that the operator $S_{1,C_0}$ corresponds to the operator $L_2$ in \cite{SW2003}. The above Proposition gives the exact and complete proof of the claim stated at page 5 of \cite{SW2003} that:

\enquote{
$$
r_e(e^{tL_1})=r_e(e^{tL_2}) \quad \forall \ t \geqslant 0.
$$
}

This equality, which corresponds exactly to the identity \eqref{C4.51} with our notation, does not result from a simple application of Theorem A of \cite{NRL1986} (see 324 of \cite{NRL1986}) as written at page 5 of \cite{SW2003}. A complementary result has to be proved for the complete proof of Theorem A in \cite{NRL1986} as explained in the Remark \ref{ThmA}.
\end{oss}

Let us now complete some further missing arguments in \cite{SW2003} as described now.
At page 5 of \cite{SW2003} the authors quote page 324 of \cite{NRL1986} to write that:

\enquote{
$$
r_e(e^{tL_2})=e^{\alpha t},
$$
}

where

\enquote{ 
$$
\alpha=\sup\{\Re(\lambda), \lambda \in \sigma(L_2)\}
$$
}

In order to determine the value of $\alpha$, the authors of \cite{SW2003} compute (see the pages 5 and 6) the set of the real part formed by the eigenvalues of $L_2$, and conclude that
$\alpha=0$. However, we recall that, the set of eigenvalues of $S$, defined by $\sigma_p(S)=:\{\lambda \in \mathbb{C},(\lambda Id_X -S) \mbox{ is not injective} \}$ is included in $\sigma(S)$. In the general case, one cannot assert without proving further properties that $\sigma_p(S)=\sigma(S)$. Hence the proof that $\alpha=\sup\{\Re(\lambda), \lambda \in \sigma(L_2)\}=0$
is not complete, even though, as we shall now prove, this identity is correct.

Thus, we complete here the proof of this last identity. For this and for coherence with our present paper, we use in the sequel our notation rather than the ones used in \cite{SW2003}, that is
$L_2$ is replaced by the notation $S_{1,C_0}$, and $\alpha$ is denoted by $\alpha_0$ as in \cite{NRL1986}.

\begin{prop}\label{spectrumS1}
We consider the operator $S_{1,C_0}$ defined in \eqref{C4.9} with $C_0$ replacing $C$, $C_0$ being defined in \eqref{C0}. We recall that $S_{1,C_0}$ is an unbounded operator in the Hilbert space $X$ defined in \eqref{C4.5}. Assuming that $b$ is a nonnegative function on $[0,l]$, we have:
\begin{equation}\label{Cb4.52}
\sigma(S_{1,C_0})=\Big\{\dfrac{k\pi }{l} \sqrt{\dfrac{K}{\rho}} i, k \in \mathbb{Z}\Big\} \cup \Big\{- \dfrac{1}{2l \,I_{\rho}} \int_0^l b(y)dy +  \dfrac{k\pi }{l} \sqrt{\dfrac{E I }{I_{\rho}}} i, k \in \mathbb{Z}\Big\},
\end{equation}
and the spectral bound of $S_{1,C_0}$ is given by:
\begin{equation}\label{C4.53}
\alpha_0=\sup\{\Re(\lambda), \lambda \in \sigma(S_{1,C_0})\}=0. 
\end{equation}
Moreover we have
\begin{equation}\label{C4.53b}
r(e^{tS_{1,C_0}})=r_e(e^{tS_{1,C_0}})=e^{\alpha_0 t}=1 \quad \forall \ t \geqslant 0
\end{equation}
where $r(e^{tS_{1,C_0}})$ stands for the spectral radius of the bounded operator $(e^{tS_{1,C_0}})$.
\end{prop}
\begin{proof}
We have $\rho(S_{1,C_0})=\{ \lambda \in \mathbb{C}, (\lambda Id_X - S_{1,C_0}) \mbox{ is invertible}\}$. Thus $\lambda \in \mathbb{C} \in \rho(S_{1,C_0})$ if and only if for all $Z=(z_1,z_2, z_3, z_4)^T \in X$, there exists a unique $Y=(p,\widehat{\varphi}, q, \psi)^T \in D(S_{1,C_0})$ such that $(\lambda Id_X - S_{1,C_0})U=Z$. 

Let  $Z=(z_1,z_2, z_3, z_4)^T$ be any arbitrary vector function in $X$. Since $\widehat{K}$ and $C_0$ are diagonal matrices, the equation $(\lambda Id_X - S_{1,C_0})U=Z$, is equivalent to $Y=(p,\widehat{\varphi}, q, \psi)^T \in D(S_{1,C_0})$ and to the set of equations to be satisfied for the components of $Y$
$$
\begin{cases}
p_x + \lambda \sqrt{\rho/K} \,p= \sqrt{\rho/K}\, z_1,\\
\widehat{\varphi}_x + \sqrt{I_{\rho}/(EI)} \Big(\dfrac{b(y)}{2I_{\rho}} + \lambda\Big)\, \widehat{\varphi}= \sqrt{I_{\rho}/(EI)} \, z_2,\\
q_x - \lambda \sqrt{\rho/K} \, q= -\sqrt{\rho/K} \, z_3,\\
\psi_x - \sqrt{I_{\rho}/(EI)} \Big(\dfrac{b(y)}{2I_{\rho}} + \lambda\Big)\, \psi= -\sqrt{I_{\rho}/(EI)} \, z_4.
\end{cases}
$$
Multiplying by an appropriate exponential term  each of the above four equations, to integrate these equations, we easily deduce that for all $x \in [0,l]$, we have:
$$
\begin{cases}
p(x )=p(0) e^{-\lambda \sqrt{\rho/K} \, x} + \sqrt{\rho/K}\, \int_0^x e^{-\lambda \sqrt{\rho/K} \, (x-s)} z_1(s)ds,\\
\widehat{\varphi}(x)= \widehat{\varphi}(0)e^{- \sqrt{I_{\rho}/(EI)} \Big(\lambda \,x + \dfrac{1}{2I_{\rho}} \int_0^x b(s)ds\Big)
} + \sqrt{I_{\rho}/(EI)} \int_0^x e^{-\sqrt{I_{\rho}/(EI)} \Big(\lambda \,(x-s) + \dfrac{1}{2I_{\rho}} \int_s^x b(\tau)d\tau\Big)
} \, z_2(s)ds,\\
q(x )=q(0) e^{\lambda \sqrt{\rho/K} \, x} - \sqrt{\rho/K}\, \int_0^x e^{\lambda \sqrt{\rho/K} \, (x-s)} z_3(s)ds,\\
\psi(x) = \psi(0) e^{\sqrt{I_{\rho}/(EI)} \Big(\lambda \,x + \dfrac{1}{2I_{\rho}} \int_0^x b(s)ds\Big)
} - \sqrt{I_{\rho}/(EI)} \int_0^x e^{\sqrt{I_{\rho}/(EI)} \Big(\lambda \,(x-s) + \dfrac{1}{2I_{\rho}} \int_s^x b(\tau)d\tau\Big)
} \, z_4(s)ds.
\end{cases}
$$
Since $(p,\widehat{\varphi}, q, \psi)^T \in D(S_{1,C_0})$, we also have
$$
(p+q)(0)=(p+q)(l)=0, (\widehat{\varphi}+\psi)(0)=(\widehat{\varphi}+\psi)(l)=0.
$$
Setting $x=l$ in the above system and adding the resulting first and third equations, as well as in a second step the resulting second and fourth equations, we get
the following system of two equations to determine the two unknowns $p(0)$ and $\widehat{\varphi}(0)$:
$$
\begin{cases}
p(0) \Big [e^{\lambda \sqrt{\rho/K} \, l} - e^{-\lambda \sqrt{\rho/K} \, l} \Big ]= \sqrt{\rho/K} \int_0^l \Big [ e^{\lambda \sqrt{\rho/K} \, (s-l)}z_1(s) - e^{-\lambda \sqrt{\rho/K} \, (s-l)}z_3(s) \Big ]ds,\\
\widehat{\varphi}(0) \Big [e^{\sqrt{I_{\rho}/(EI)} \Big(\lambda \,l + \dfrac{1}{2I_{\rho}} \int_0^l b(s)ds\Big)} - e^{-\sqrt{I_{\rho}/(EI)} \Big(\lambda \,l + \dfrac{1}{2I_{\rho}} \int_0^l b(s)ds\Big)
}\Big ]= \\
\sqrt{I_{\rho}/(EI)} \int_0^l \Big [
e^{\sqrt{I_{\rho}/(EI)} \Big(\lambda \,(s-l) - \dfrac{1}{2I_{\rho}} \int_s^l b(\tau)d\tau\Big)} z_2(s) -
e^{-\sqrt{I_{\rho}/(EI)} \Big(\lambda \,(s-l) - \dfrac{1}{2I_{\rho}} \int_s^l b(\tau)d\tau\Big)} z_4(s)
\Big ]ds.
\end{cases}
$$
The two unknowns $p(0)$ and $\widehat{\varphi}(0)$ can be uniquely determined from the above system of two equations if and only if $\lambda \in \mathbb{C}$ satisfies:
\begin{equation}\label{Cb4.54}
e^{2\lambda \sqrt{\rho/K} \, l} \neq 1 \mbox{ and } e^{\sqrt{2I_{\rho}/(EI)} \Big(\lambda \,l + \dfrac{1}{2I_{\rho}} \int_0^l b(s)ds\Big)} \neq 1.
\end{equation}
For all $\lambda$ satisfying \eqref{Cb4.54}, one can easily check that the above computations allow us for all $Z \in X$ to determine a unique $U \in D(S_{1,C_0})$ solving
$(\lambda Id_X - S_{1,C_0})U=Z$. 
Thus the resolvent set of the operator $S_{1,C_0}$ is given by
$$
\mathbb{C} \backslash (\Big\{\dfrac{k\pi }{l} \sqrt{\dfrac{K}{\rho}} i, k \in \mathbb{Z}\Big\} \cup \Big\{- \dfrac{1}{2l \,I_{\rho}} \int_0^l b(y)dy +  \dfrac{k\pi }{l} \sqrt{\dfrac{E I }{I_{\rho}}} i, k \in \mathbb{Z}\Big\}),
$$
which proves our claim \eqref{Cb4.52}. Since $b$ is nonegative on $[0,l]$, we easily deduce that \eqref{C4.53} holds. 

Thanks to respectively Theorem B (see pages 324-325), to Lemma 5 (see page 336) of \cite{NRL1986}, applied with $\widehat{K}$ (replacing the corresponding matrix $K$ of \cite{NRL1986}), $C_0$, the matrix $\widehat{E}$ (replacing the corresponding matrix $E$ of \cite{NRL1986}), and the matrices $D, F, G$ as defined in the present paper and with the operator $S_{1,C_0}$ (replacing the operator $T_4$ in \cite{NRL1986}), we have
$$
r(e^{tS_{1,C_0}})=r_e(e^{tS_{1,C_0}})=e^{\alpha_0 t}=1 \quad \forall \ t \geqslant 0.
$$
This concludes the proof.
\end{proof}

\begin{oss}\label{thmB}
The above Proposition and its proof complete the gap in the proof of the claimed identity $\alpha=0$ in \cite{SW2003}.
\end{oss}
We can now state a complete proof of the non exponential stability of the Timoshenko system denoted by (1.3) in \cite{SW2003} as quoted at the beginning of Section 2 in the present paper.
\begin{thm}\label{NES}
Assume that \eqref{DS} holds and that $b$ is a nonnegative function on $[0,l]$. Then the Timoshenko system \eqref{C4.55} is non exponentially stable.
\end{thm}
\begin{proof}
The argument stated in \cite{SW2003} that $r_e(e^{tL})=r_e(e^{tL_1})$ for all $t \geqslant 0$ is correct ($L-L_1$ is a compact operator in the energy space $\mathcal{H}$).
Using the complete proofs we gave in the present paper, we have thanks to Proposition \ref{P2C0} and to Proposition \ref{spectrumS1}:
$$
r_e(e^{tL_1})=r_e(e^{tS_{1,C_0}})=1 \quad \forall \ t \geqslant 0.
$$
Hence we have  $r_e(e^{tL})=1$ for all $ t \geqslant 0$. Since $r(e^{tL}) \geqslant r_e(e^{tL})=1$ as recalled in Remark \ref{G} (see \eqref{relr}), we have $r(e^{tL}) \geqslant 1$ for all $t \geqslant 0$.
We now denote by  $\omega_0$ the growth bound of the strongly continuous semigroup $(e^{tL})_{t \geqslant 0}$, defined as recalled in \eqref{OMG} by:
$$
\omega_0=\inf_{t>0}\Big(\dfrac{\log{\big(||e^{tL}||_{L(\mathcal{H})}\big)}}{t}\Big),
$$
Using the classical results recalled in Remark \ref{StandardDef} (see \eqref{SRW}), we have
$$
r(e^{tL})=e^{\omega_0 t} \geqslant 1, \quad \forall \ t \geqslant 0.
$$
Hence $\omega_0 \geqslant 0$, which proves that $(e^{tL})_{t \geqslant 0}$ is not exponentially stable.
\end{proof}
\begin{oss}\label{w0=0}
Indeed, since $b$ is nonnegative over $[0,l]$, it is easy to check that $(e^{tL})_{t \geqslant 0}$ is a semigroup of contractions (this is also recalled in \cite{SW2003}), so that $\omega_0 \leqslant 0$. This, together with the above inequality imply that $\omega_0 =0$.

Note also that the argument that $(e^{tL}) \geqslant 1$ and thus $(e^{tL})_{t \geqslant 0}$ is not exponentially stable is also used in \cite{SW2003}, however even for the last part, the proof of the computation of the spectral bound of $S_{1,C_0}$ is not complete, the authors having computed only the punctual spectrum of $S_{1,C_0}$, without checking that it is real part coincides with the real part of the full spectrum.
\end{oss}

\section{Open access to \cite{SW2003} and submission history}

The article \cite{SW2003} is available in open access at the link: 
\vskip 0.2 cm
\url{https://ejde.math.txstate.edu/Volumes/2003/29/soufyane.pdf}
\vskip 0.2 cm
The pdf of the article starts at page 1 and ends at page 14.
\vskip 0.2 cm
The history of submission is indicated at the bottom of page 1 as follows:
\vskip 0.2cm 
\enquote{Submitted September 2, 2002. Published March 16, 2003}
\vskip 0.2 cm

More information on citations and disseminations of \cite{SW2003} can be found in \cite{AB2023I}.

\section{Conclusion}
In this part II, we give rigorous and correct proofs of the strong stability of Timoshenko beams by a single linear feedback in the second equation, and of the non uniform stability when the speeds of propagation are distinct. For this, we correct several arguments of \cite{SW2003}, introduce precise (absent) successive functional frames and state and prove several new properties and results. We also complete the proof of Theorem A in \cite{NRL1986} by stating and proving Proposition \ref{UptoPi} (see also Proposition \ref{UptoPiB} for a more general formulation).

\end{document}